\documentclass[reqno]{amsart}

\usepackage{amsfonts}
\usepackage{amsthm}
\usepackage{amsmath}
\usepackage{graphicx} 
\DeclareGraphicsExtensions{.png,.pdf,.jpg,.ps,.eps}
\usepackage[all]{xy}
\usepackage[active]{srcltx} 
\usepackage[latin1]{inputenc}
\usepackage{verbatim}
\usepackage{amssymb}
\usepackage{mathrsfs}
 \usepackage[]{todonotes}
\input xy
\xyoption{all}

\newtheorem{thm}{Theorem}[section]
\newtheorem{cor}[thm]{Corollary}
\newtheorem{prop}[thm]{Proposition}
\theoremstyle{definition}
\newtheorem{defn}[thm]{Definition}
\theoremstyle{definition}
\newtheorem{exa}[thm]{Example}
\theoremstyle{definition}
\newtheorem{rem}[thm]{Remark}
\newtheorem{lem}[thm]{Lemma}

\numberwithin{equation}{section}
\newcommand{\bndkl}{\mathscr{B}un_X^{k,l}(n,d)}
\newcommand{\bnd}{\mathscr{B}un_X(n,d)}
\newcommand{\bs}{\mathscr{B}un^{s}}

\newcommand{\bp}{\mathscr{B}un^{\mathfrak{P}}_X(n,d)}

\newcommand{\df}{deg (F)}

\newcommand{\Ql}{{{\overline{\mathbb Q}}_l}}

\def\Q{{\mathbb Q}}

\def\C{{\mathbb C}}

\def\F{{\mathbb F}}

\def\Spec{\mathrm{Spec}}

\begin{document}

\title[Geometry of moduli stacks of $(k, l)$-stable bundles]{Geometry of moduli stacks of $(k, l)$-stable vector bundles over algebraic curves}

\author[O. Mata-Guti\'errez ]
{O. Mata-Guti\'errez }
\address{Departamento de Matemáticas, CUCEI,  Universidad de Guadalajara\\
Av. Revoluci\'on 1500,
C. P. 44430 \\
Guadalajara, Jalisco, M\'exico}
\email{osbaldo.mata@academico.udg.mx, osbaldo@cimat.mx}

\author[F. Neumann]
{Frank Neumann}
\address{Department of Mathematics\\
University of Leicester\\
University Road, Leicester LE1 7RH, England, UK}
\email{fn8@le.ac.uk}



\begin{abstract}
{We study the geometry of the moduli stack of vector bundles of fixed rank and degree over an algebraic curve
by introducing a filtration made of open substacks build from $(k, l)$-stable vector bundles. The concept of $(k, l)$-stability was introduced
by Narasimhan and Ramanan to study the geometry of the coarse moduli space of stable bundles. We will exhibit the stacky picture
and analyse the geometric and cohomological properties of the moduli stacks of $(k, l)$-stable vector bundles. For particular pairs $(k, l) $
of integers we also show that these moduli stacks admit coarse moduli spaces and we discuss their interplay.}
\end{abstract}
\maketitle

\vspace*{-0.5cm}

\section*{Introduction}

Let $X$ be a geometrically irreducible smooth projective algebraic curve of genus $g\geq 2$ over either the field $\C$ of complex numbers or the algebraic closure $\overline{\F}_q$ of the field $\F_q$ with $q=p^s$ elements for a prime $p$. Using Geometric Invariant theory, Mumford \cite{M} constructed a coarse moduli space $M_{X}^{s}(n,d)$ for the moduli problem of stable vector bundles of rank $n$ and degree $d$ over X and showed that this moduli space is in fact a non-singular quasi-projective scheme of dimension $n^{2}(g-1)+1$. If in addition the rank $n$ and the degree $d$ are actually coprime, this moduli space is in fact a projective scheme and a fine moduli space.
More generally, considering the notion of S-equivalence classes of vector bundles, Seshadri \cite{Se2} constructed a coarse moduli space $M_{X}^{ss}(n,d)$ for semistable vector bundles of rank $n$ and degree $d$, which gives a natural compactification of the moduli space $M_{X}^{s}(n,d)$ of stable bundles over $X$.

Later, Narasimhan and Ramanan \cite{NR-def, NR-Geo} introduced a more general concept of $(k,l)$-stability for vector bundles over an algebraic curve $X$ defined for any pair $(k, l)$ of integers, which refines the classical notion of stability. A vector bundle $E$ is hereby $(k, l)$-stable if for any proper subbundle $F$ of $E$ we have for the generalised slopes $\mu_k(F)<\mu_{k-l}(E)$, where for a given pair $(k, l)$ the generalised slope is defined as $\mu_{k-l}(E)=(deg(E)+k-l)/rk(E)$. Narasimhan and Ramanan \cite{NR-Geo} derived conditions for some special values of integers $k$ and $l$ for which $(k, l)$-stable bundles over $X$ exist and proved some fundamental properties of $(k, l)$-stability, among them openness.
In particular, they used $(k, l)$-stability for the special pairs $(0, 1)$, $(1, 0)$ and $(1, 1)$ to define an open set inside the moduli space $M_X^{s}(n, L)$ of stable bundles over $X$ with fixed determinant $L$ that allows for the construction of a Hecke correspondence and an associated space of Hecke cycles inside a certain Hilbert scheme associated to $M_X^{s}(n, L)$, which under certain conditions gives a non-singular model for $M_X^{s}(n, L)$.
This Hecke correspondence has also been used recently in many other ways to study the geometry of the moduli space $M_{X}^{s}(n,d)$ of stable bundles over $X$ (see \cite{Bra-Mat, NH}).

In this article we embark to study the general moduli problem for $(k,l)$-stable vector bundles of rank $n$ and degree $d$ over an algebraic curve $X$ for {\it any} pair $(k, l)$ of integers. In the first section we will derive some general theorems (Theorem \ref{prop-nonempti} and Proposition \ref{thm-necsuf}) establishing conditions for the existence of $(k, l)$-stable vector bundles over $X$ for general pairs $(k, l)$ of integers and hereby extending the particular existence results of Narasimhan and Ramanan in \cite{NR-Geo}.  In section two we address the general moduli problem for $(k, l)$-stable vector bundles and analyse under which conditions with respect to the choice of integers $k, l, n, d$ the associated moduli functor is representable or corepresentable. It turns out that if the pair $(k, l)$ of integers meets the conditions that $0\leq k(n-1)+l < (n-1)(g-1)$ and $0\leq k+l(n-1)< (n-1)(g-1)$ then the coarse moduli space $M_X^{k,l}(n,d)$ for $(k, l)$-stable vector bundles over $X$ exists as an open subscheme of the moduli space $M_X^{s}(n,d)$ of stable vector bundles. The third section exhibits a general discussion of the set of isomorphism classes of $(k, l)$-stable vector bundles over $X$ for any pair
$(k, l)$ of integers, where among other things filtrations between the different sets of isomorphism classes are derived and how they relate to the coarse moduli spaces constructed before.  This allows for further characterisations of $(k, l)$-stable vector bundles.
In the fourth section we introduce the moduli stack $\bndkl$ of $(k, l)$-stable vector bundles of rank $n$ and degree $d$ over the algebraic curve $X$ for any pair $(k, l)$ of integers and study its basic geometric properties. It turns out that it is an Artin stack, which is locally of finite type and has an open embedding in the moduli stack $\bnd$ of all vector bundles of rank $n$ and degree $d$ over $X$ (Theorem \ref{thm-stack}). We also establish particular filtrations of the moduli stack $\bnd$ by means of open substacks of $(k, l)$-stable bundles:
$$
\cdots\subset \mathscr{B}un_X^{k-3,l}(n,d)\subset \mathscr{B}un_X^{k-2,l}(n,d)\subset \mathscr{B}un_X^{k-1,l}(n,d)\subset \mathscr{B}un_X^{k,l}(n,d)\subset\cdots
$$
$$
\cdots\subset \mathscr{B}un_X^{k,l-3}(n,d)\subset \mathscr{B}un_X^{k,l-2}(n,d)\subset \mathscr{B}un_X^{k,l-1}(n,d)\subset \mathscr{B}un_X^{k,l}(n,d)\subset\cdots
$$
In section five we then carefully analyse the relations between the moduli stacks and the coarse moduli spaces of $(k, l)$-stable vector bundles with respect to the conditions under which these coarse moduli spaces do exist. Finally, in the last section we derive some cohomological properties of the moduli stacks $\bnd$ and in particular discuss the rank $2$ case. We end by discussing a general Hecke correspondence involving the moduli stacks $\bnd$ by using appropriate Grassmannian bundles of the universal bundles over the moduli stacks involved. In this way we extend the approach of Narasimhan and Ramanan in \cite{NR-Geo} to the general case.

{\it Notation and conventions.} All schemes will be considered over the base $\Spec(\F)$, where $\F$ is either the field $\C$ of complex numbers or the algebraic closure $\F=\overline{\F}_q$ of the finite field $\F_q$ of characteristic $p$ with $q=p^s$ elements for a prime number $p$. The category of schemes $Sch/\Spec(\F)$ over $\Spec(\F)$ will be endowed with the \'etale topology whenever we need to emphasise a site.

\section{Vector bundles over algebraic curves, Segre invariants and $(k,l)$-stability.}

Let $X$ be an irreducible smooth projective algebraic curve of genus $g\geq 2$ over $\Spec(\F)$, where $\F$ is either the field of complex numbers $\C$ or the algebraic closure $\F=\overline{\F}_q$ of the field $\F_q$.

Narasimhan and Ramanan in \cite{NR-Geo} introduced the notion of $(k,l)$-stability and $(k, l)$-semistability for
vector bundles over $X$ and showed that $(k,l)$-stability is an open property for vector bundles over $X$ (see \cite[Proposition 5.3]{NR-Geo}).
Following Narasimhan and Ramanan we define (see \cite[Definition 5.1]{NR-Geo}):

\begin{defn}
Let $(k,l)$ be a pair of integers and $E$ a vector bundle over $X$. We define the {\it generalised slope} as the rational number
$$\mu_{k-l}(E)=\frac{deg(E)+k-l}{rk(E)},$$
and say that the vector bundle $E$ over $X$
is $(k,l)${\it -stable} (resp.{\it $(k, l)$-semistable}) if for any subbundle $F$ of $E$, we have
\begin{equation}\label{equivalenciaskl}
\mu_{k}(F) <\mu_{k-l}(E) \,\,(\mathrm{resp.}\,\,\mu_{k}(F)\leq \mu_{k-l}(E)).
\end{equation}
\end{defn}
Criteria for the existence of $(k, l)$-stable vector bundles for the pairs $(0,1)$, $(1,0)$ and $(1,1)$ were given by Narasimhan and
Ramanan in \cite[Proposition 5.4]{NR-Geo}. In Theorem \ref{prop-nonempti} below, we will extend this result for any pair $(k,l)$ of integers.

Obviously, $(0, 0)$-stability (resp.$(0, 0)$-semistability) just gives the classical notion of stability (resp.~semistability)
for vector bundles over algebraic curves. It is also an easy consequence from the definition,
that if $E$ is a $(k, l)$-stable vector bundle and $L$ a line bundle, then $E\otimes L$ is $(k,l)$-stable
and the dual vector bundle $E^*$ is $(l, k)$-stable.
A vector bundle of degree $0$ is stable if and only if it is $(0, 1)$-stable and a vector bundle
of degree $1$ is stable if and only if it is $(0, 1)$-semistable (see \cite[Remark 5.2]{NR-Geo}.

We also have the following fundamental properties for $(k, l)$-stability of vector bundles:

\begin{prop}[Narasimhan-Ramanan]
Let $(k,l)$ be a pair of integers. Then we have the following:
\begin{itemize}
\item[(1)] $(k, l)$-stability is an open property.
\item[(2)] If $E$ is $(k, l)$-stable, then $E$ is also $(k, l-1)$-stable and $(k-1, l)$-stable.
\item[(3)] Given an exact sequence of locally free sheaves
$$ 0\to E'\to E\to \mathcal{O}^{r}_{X}\to 0,$$
it follows that if $E$ is $(k, l)$-stable, then $E'$ is $(k, l-r)$-stable.
\end{itemize}
\end{prop}
\begin{proof} The first property (1) is basically \cite[Proposition 5.3]{NR-Geo}.
Property (2) is a direct consequence of the definition. Finally, the last property (3) is a consequence of
\cite[Lemma 5.5]{NR-Geo} and the argument goes as follows: Let $F$ be a proper subbundle of $E'$
and $\bar{F}$ the saturation of $F$ in $E$. Then $\bar{F}$ is a proper subbundle of $E$ and
therefore $\mu_{k}(\bar{F})<\mu_{k-l}(E)$. Moreover, $\mu_{k}(F)\leq\mu_{k}(\bar{F})<\mu_{k-l}(E)=\mu_{k-l+r}(E').$ This is our assertion.
\end{proof}

Now we will recollect some general properties of Segre invariants (see \cite{LN-Max, L-some, BL-stra,Teix-ona}), which we will need to use later.

\begin{defn}
Let $E$ be vector bundle over $X$ of rank $n$ and degree $d.$
Let $m\in \mathbb{Z}$ such that $1\leq m\leq n-1$. The
 \textit{m-Segre invariant} for $E$, is denoted by $s_{m}(E)$ and defined as the integer
$
s_{m}(E)= md-n \cdot \df_{\max},
$
where $F_{\max}\subset E$ is a proper subbundle of rank $m$ and maximal degree.
\end{defn}

Hirschowitz proved in \cite{Hirs-pro} the following fundamental inequality
\begin{equation}\label{cota-H-sm}
s_{m}(E)\leq m(n-m)(g-1)+(n-1).\\
\end{equation}

\noindent Specifically, he proved that there is a unique integer
$\delta_{m}$ with  $0\leq \delta_{m} \leq n-1$
and $m(n-m)(g-1)+\delta_{m}\equiv md\mod n,$ such that
\begin{equation}\label{igual-H-sm}
s_{m}(E)\leq m(n-m)(g-1)+\delta_{m}.
\end{equation}
\noindent Equality holds if $E$ is general.

Let $M^s_X(n,d)$ be the set of all stable vector bundles of rank $n$ and degree $d$ over $X$.
Furthermore, the set of all stable vector bundles of rank $n$ and degree $d$ with
$m$-Segre invariant equal to $s$ will be denoted by $M^s_X(n,d,m,s),$ that is
$
M^s_X(n,d,m,s):=\{E\in M^s_X(n,d)|\ s_{m}(E)=s\}.
$
If $s$ is such that $0<s\leq m(n-m)(g-1)$, $s\equiv md\mod n$ and
$g\geq 2,$ then $M^s_X(n,d,m,s)$  is non-empty and
irreducible of dimension $n^{2}(g-1)+1+s-m(n-m)(g-1)$
(see \cite{Teix-ona, BL-stra}).

\begin{rem}\label{kl-vs-st}
From the definition of the Segre invariant and the definition of $(k,l)$-stability
we can see that $E$ is $(k,l)$-stable if and only if $s_{m}(E)>k(n-m)+ml$
for all $m$ with $1\leq m\leq n-1$.
\end{rem}

\begin{rem}\label{A}
  Suppose that $E$ is a stable vector bundle of rank $n$ and degree $d$.
  Also suppose that $E$ is not $(k,l)$-stable. Then there exists
  a proper subbundle $F\subset E$ such that
  $$
  \mu_{k}(F)\geq \mu_{k-l}(E)
  $$
  and an exact sequence
  \begin{eqnarray}\label{exacseq}
  0\rightarrow F \rightarrow E \rightarrow E/F \rightarrow 0.
  \end{eqnarray}
  Let now $rk(F)=m$ and $deg(F)=\delta$. Then we have
  $
  (\delta+k)/m\geq (d+k-l)/n
  $
  and
  $
  d/n>\delta/m.
  $\\
\end{rem}

  Now, applying \cite[Proposition 2.6]{NR-def} to $F$ and $(m,\delta)$
  we get a family $\mathcal{F}$ of vector bundles
  on $X$ of rank $m$ and degree $\delta$ parametrised by a scheme $R$ with the following
  properties:
  \begin{enumerate}
    \item $R$ is irreducible, 
    \item the family $\mathcal{F}$ contains $F$ and all stable vector bundles
    of rank $m$ and degree $\delta$ on $X$.
  \end{enumerate}

Furthermore, let $\mathcal{G}$ be the family of vector bundles on $X$ of rank $n-m$ and degree $d-\delta$ parameterised by a scheme $S$ obtained by applying \cite[Proposition 2.6]{NR-def} to $E/F$ and $(n-m,d-\delta)$.

  Now let $H\subset R\times S$ be the open subscheme given such that
  $(r,s)\in H$ if $H^{0}(X, Hom( \mathcal{G}_{s},\mathcal{F}_{r}))=0.$
  Then
  $\mathcal{R}^{1}_{p_{R\times S}}(X\times R\times S, Hom(p_{13}^{*}\mathcal{G}, p_{12}^{*}\mathcal{F}))$ is locally free on $H$.
  
  Note that $H$ is non-empty, because $Hom(E/F, E)=0$. Indeed, if such an homomorphism $f\in Hom(E/F, F)$ would exist, it would give, by composition 
a non-zero homomorphism $E\twoheadrightarrow E/F \rightarrow F \hookrightarrow E$ , which is not an isomorphism. But this is impossible, since $E$ is stable.

  We set $P:=\mathbb{P}(\mathcal{R}^{1}_{p_{R\times S}}(X\times R\times S, Hom(p_{13}^{*}\mathcal{G}, p_{12}^{*}\mathcal{F})))|_{H}$ and let $\pi:P\rightarrow H$ be the projection. Then by \cite{NR-def} (see also \cite[Lemma 2.4]{Ra})
  we have the exact sequence
  $$
  0\rightarrow \pi^{*}p_{X\times R}^{*}\mathcal{F}\otimes p_{P}^{*}\tau_{P}\rightarrow \mathcal{E}\rightarrow \pi^{*}p_{X\times S}^{*}\mathcal{G} \rightarrow 0
  $$
 on $X\times P$, where $\mathcal{E}$ is a family of vector bundles parameterised by $P$ and $\tau_{P}$ the tautological hyperplane bundle.
 Now let $P^{st}$ be the open subscheme given by the stability condition, i.e.,  $q\in P^{st}$ if and only if $\mathcal{E}|_{X\times\{q\}}$ is stable.
 Moreover, $P^{st}$ is non-empty since we have the extension
 (\ref{exacseq}) in Remark \ref{A} defining a point on $P^{st}$.  Therefore, we obtain a map $\theta_{\mathcal{E}}:P^{st}\rightarrow M^s_X(n,d)$, the classifying map.
 Note that $\theta_{\mathcal{E}}(P^{st})$ is the set of  stable vector bundles of rank $n$ and degree $d$ with a subundle $F$ of rank $m$ and degree $\delta$ as considered above.

\begin{cor}\label{B}
 Let $\mathcal{F},\mathcal{G}, R$ and $S$ be as above.
 Let $H'\subset R\times S$ be the open subscheme defined as
 $$
 H'=\{(r,s)\in R\times S \  |\  \mathcal{F}_{s},\mathcal{G}_{r}\text{ are stable}\}.
 $$
 Then $\theta_{\mathcal{E}}(P^{st}|_{H'})$ is dense in $\theta_{\mathcal{E}}(P^{st}).$
 \end{cor}
\begin{proof}

Because there are no non-zero homomorphisms between two stable bundles if the first bundle has higher
slope than the second it follows that $H'\subset H$. Therefore, we can restrict the projective bundle $P^{st}$ on $H$ to $H'$. 
Furthermore, $H'$ is non-empty by construction of $R$ and $S$. In addition, as $H'$ is non-empty and $S,R$ are irreducible, $H'$ is also irreducible and dense
in $R\times S$. Therefore, $P^{st}|_{H^{'}}$ is dense in $P^{st}$ and the Corollary follows (see also \cite[Proposition 6.7]{NR-def}).
\end{proof}

\begin{rem} Note that if $E$ is a stable vector bundles of $rk(E)=n$ and $deg(E)=d$
and if we suppose that $F$ is a subbundle of $E$ of $rk(F)=m$ and $deg(F)=\delta$,
then Corollary \ref{B} implies that the exact sequence
$$
0\rightarrow F\rightarrow E\rightarrow E/F\rightarrow 0
$$
determines a point $q\in \theta_{\mathcal{E}}(P^{st})$ and such a point is in the closure
$\theta_{\mathcal{E}}(P^{st}|_{H'})$.\\
\end{rem}

The following theorem gives conditions on the general existence of $(k, l)$-stable vector bundles and under which conditions $(k, l)$-stability implies stability (see also \cite{Bra-Mat}).

\begin{thm} \label{prop-nonempti}
  Let $X$ be a non-singular projective curve of genus $g\geq 2$ and let $k,l,n$ be integers. Then:
\begin{enumerate}
\item If \begin{equation}
k(n-1)+l < (n-1)(g-1) \label{cond1},
\end{equation}
and
\begin{equation}\label{cond2}
k+l(n-1)< (n-1)(g-1),
\end{equation}
\noindent then there exist $(k,l)$-stable vector bundles of rank $n$ and degree $d$ over $X$.
\item If
\begin{equation}\label{cond4}
k(n-1)+l\geq (n-1)g,
\end{equation}
or
\begin{equation}\label{cond5}
k+l(n-1)\geq (n-1)g,
\end{equation}
then there do not exist $(k,l)$-stable vector bundles of rank $n$ and degree $d$ over $X$.

  \end{enumerate}
\end{thm}

\begin{proof}
(1) Assuming the inequalities for $k$ and $l$, we will prove that
there exist stable vector bundles that are $(k,l)$-stable.
Let $E$ be a stable vector bundle of rank $n$ and degree $d$, which
is not $(k,l)$-stable. Thus, by Remark \ref{A} there exists a proper subbundle $F\subset E$
of rank $m$ and degree $\delta$, such that
\begin{equation}\label{cond3}
\mu_{k-l}(E)\leq \mu_{k}(F).
\end{equation}

\noindent Considering the extension
$
0\to F\to E \to E/F \to 0,
$
we can assume by Corollary \ref{B} that $F$ and $E/F$ are stable (see also \cite[Proposition 2.6 ]{NR-def} and \cite[Proposition 5.4]{NR-Geo}).
Using \cite[Proposition 2.4]{NR-def}, as $\dim\,M^s_X(n,d)=n^{2}(g-1)+1$, it follows that the number of such extensions is bounded by
$ m^{2}(g-1)+1+(n-m)^{2}(g-1)+1+h^{1}\left((E/F)^{*}\otimes F\right)-1=
(n^{2}-mn+m^{2})(g-1)+1 + dm-n\delta.$
We will show now that this number is actually less than $n^{2}(g-1)+1$.
First, by (\ref{cond1}) and (\ref{cond2}) we have that
$k(n-m)+ml<m(n-m)(g-1)$ and by (\ref{cond3}) $dm- n\delta\leq k(n-m)+ml.$
Thus, $dm-n\delta<m(n-m)(g-1),$ which implies
$(n^{2}-mn+m^{2})(g-1)+1 + dm-n\delta<n^{2}(g-1)+1,$ i.e.,
the dimension of the locus of stable vector bundles
satisfying (\ref{cond3}) is less than $\dim M^s_X(n,d)$.
Allowing $m$ to vary with values $1\leq m\leq n-1$,
we conclude that the dimension of the locus of non-$(k,l)$-stable vector bundles
is also less than $\dim M^s_X(n,d)$.

(2) Assuming that a pair of integers $(k_{0},l_{0})$ satisfies condition (\ref{cond4}), we will prove that there is no
vector bundle which is $(k_{0}, l_{0})$-stable. Let $E$ be a vector bundle of rank $n$ and degree $d$ and
let $L_{0}\subset E$ be a line subbundle of maximal degree. By
(\ref{cota-H-sm}) and (\ref{cond4}) we obtain that
$
d - n\cdot deg(L_{0})=s_{1}(E)\leq (n-1)g \leq k_{0}(n-1)+l_{0}.
$
\noindent This implies that $\mu_{k_{0}}(L_{0})\geq \mu_{k_{0}-l_{0}}(E),$
and therefore $E$ is a non-$(k_{0},l_{0})$-stable vector bundle.
Now suppose that the pair of integers $(k_{0},l_{0})$ satisfies condition (\ref{cond5}), then
we consider a subbundle $F\subset E$ of rank $n-1$ and
maximal degree and the rest of the proof goes just as before.
\end{proof}

Finally, we give a necessary and sufficient general condition for the existence of $(k, l)$-stable vector bundles over an algebraic curve $X$.

\begin{prop}\label{thm-necsuf}
Let $X$ be a non-singular projective curve of genus $g\geq 2$, $n$ be a positive integer and $(k,l)$ be any pair of integers.
Then there exist $(k,l)$-stable vector bundles of rank $n$ and degree $d$ if and only if the pair $(k,l)$ satisfies the inequality
\begin{eqnarray}\label{eq-ssi}
k(n-m)+ ml <m(n-m)(g-1)+\delta_{m},
\end{eqnarray}
for all integers $m$ with $1\leq m\leq n-1.$
\end{prop}

\begin{proof}
Suppose that $E$ is a $(k,l)$-stable vector bundle of rank $n$ and degree $d$.
Combining (\ref{equivalenciaskl}) and (\ref{igual-H-sm}), we obtain that
$$k(n-m)+ml<s_{m}(E)\leq m(n-m)(g-1)+\delta_{m}$$
for all $m$ and this implies (\ref{eq-ssi}).

Conversely, let the pair $(k,l)$ satisfy the inequality (\ref{eq-ssi}),
for all $m$ with $1\leq m\leq n-1$.
Then by (\ref{igual-H-sm})
the general vector bundle $E$ has an $m$-Segre invariant given by
$s_{m}(E)= m(n-m)(g-1)+\delta_{m}$
for all $m$ (see \cite{Hirs-pro}).
It follows therefore that $E$ is $(k,l)$-stable by using (\ref{eq-ssi}) and (\ref{equivalenciaskl}).
This completes the proof.
\end{proof}

Let $X$ be a non-singular projective curve of genus $g\geq 2$, $n$ be a positive integer and $(k,l)$ be any pair of integers. Furthermore, let
 \begin{equation}\label{ec1}
0\leq k(n-1)+l<(n-1)(g-1),
\end{equation}
\begin{equation}\label{ec2}
0 \leq k+l(n-1)<(n-1)(g-1).
\end{equation}

Under these conditions, if $E$ is $(k,l)$-stable,
then the left parts $0\leq k(n-1)+l$ and $0\leq k+l(n-1)$ of the above inequalities imply that $E$ is in fact stable.

Hence there always exist $(k,l)$-stable vector bundles over $X$, which are also stable.
Thus, if the pair of integers $(k,l)$ satisfies the above inequalities, then $(k,l)$-stability determines an open subscheme $M_X^{k,l}(n,d)$
parameterising the $(k, l)$-stable vector bundles inside the moduli space $M_X^{s}(n,d)$ of stable vector bundles over $X$ as $(k, l)$-stability
is an open property. The codimension of this locus can be determined as follows.

\begin{thm}\label{codimprop}
Let $k,l$ be integers such that
$0\leq k(n-1)+l\leq (n-1)(g-1)$ and
$0\leq k+l(n-1)\leq (n-1)(g-1)$.
Then,
$$
codim\ (M_X^s(n,d)\setminus M_X^{k,l}(n,d))\geq \min \left\{\begin{array}{c}
                                         (n-1)(g-1)-k(n-1)-l, \\
                                         (n-1)(g-1)-k-l(n-1)
                                       \end{array}
\right\}.
$$
\end{thm}

\begin{proof}
  Let $E\in M_X^s(n,d)\setminus M_X^{k,l}(n,d)$ be a vector bundle such that there exists a subbundle $F\subset E$ of rank $m$
and degree $\delta$ which satisfies $\mu_{k-l}(E)\leq\mu_{k}(F)$.
 We have, as in the proof of Theorem
\ref{prop-nonempti},  that the dimension of such stable vector bundles is
$(n^{2}-nm+m^{2})(g-1)+1+dm -n\delta$.
Moreover, this number is bounded above by
$(n^{2}-nm+m^{2})(g-1)+1+(n-m)k+ml.$ Thus,
$
\dim M(n,d)-\dim (M_X^s(n,d)\setminus M_X^{k,l}(n,d))\geq (nm-m^{2})(g-1)-(n-m)k-ml.
$
\noindent Considering $m$ as a parameter variable, we can see that the maximum of
$(nm-m^{2})(g-1)-(n-m)k+ml$ is obtained
whenever $m=1$ or $m=n-1$. Consequently, the codimension of $M_X^s(n,d)\setminus M_X^{k,l}(n,d)$ is bounded below
by $
\min \{(n-1)(g-1)-k(n-1)-l, \ (n-1)(g-1)-k-l(n-1)\}.$
This gives the desired conclusion.
\end{proof}

\section{Moduli spaces of $(k,l)$-stable vector bundles over an algebraic curve}

In this section we will study the moduli problem and the associated moduli functor for $(k,l)$-stable vector bundles over an algebraic curve.
Though this moduli problem is similar to the moduli problem of stable vector bundles we get a refinement and filtration as we can vary the pair $(k, l)$ of integers.

First, we will need to introduce the notion of families of $(k, l)$-stable vector bundles over an algebraic curve and an adequate equivalence relation among them.

\begin{defn} Let $X$ be a smooth projective algebraic curve and let $T$ be a scheme over $\Spec(\F)$.  A {\it family of $(k,l)$-stable vector bundles of rank $n$
and degree $d$ over $X$} parametrised by $T$ is a vector bundle $E$ over $X\times T$
such that for each point $t$ of $T$, the restriction $E_{t}$ is a $(k,l)$-stable vector bundle of rank $n$ and degree $d$
over $X$.
\end{defn}

We define an equivalence relation for families of $(k,l)$-stable vector bundle over $X$ as follows:
Two families $E$ and $E'$ of $(k, l)$-stable vector bundles parametrised by the scheme $T$ are {\it equivalent},
denoted by $E\sim E'$, if there exists a line bundle over $T$ such that $E\otimes p^{*}_{2}L$ and
$E'$ are isomorphic, where $p^{*}_{2}L$ is the pullback of $L$ along the projection morphism $p_2: X\times T\rightarrow T$.

Observe that, when $(k,l)=(0,0)$ this is precisely the equivalence relation normally considered for stable vector bundles over algebraic curves.\\

Let us now consider the moduli functor for $(k, l)$-stable vector bundles over $X$
$$\mathcal{M}_X^{k, l}(n, d): (Sch/\Spec(\F))^{op}\rightarrow Sets,$$
which associates to any scheme $T$ the set $\mathcal{M}_X^{k,l}(n,d)(T)$ of equivalence classes of
families of $(k,l)$-stable vector bundles and to any morphism of schemes $f: T'\rightarrow T$ the map of sets
$f^*: \mathcal{M}_X^{k,l}(n,d)(T)\rightarrow \mathcal{M}_X^{k,l}(n,d)(T')$ induced via the pullback operation.

We will study the representability and corepresentability of the moduli functor functor $\mathcal{M}_X^{k,l}(n,d)$, or in other words the existence of a fine or coarse moduli space for $(k, l)$-stable vector bundles over $X$.  This will depend on the rank $n$ and degree $d$ as in the case of stable vector bundles,
but in addition also on the paricular pair $(k,l)$ of integers. We have to consider two general cases.
In the first case we will assume that the pair $(k,l)$ of integers satisfies the inequalities (\ref{ec1}) and (\ref{ec2}).
In the second case we consider a more general situation, namely when for the pair $(k,l)$ of integers we have that
$k(n-1)+l<0$ or $k+l(n-1)<0$.

In the first case, the representability or corepresentability
of the moduli functor is basically a consequence of the representability or corepresentability of the moduli functor
for stable vector bundles over the algebraic curve $X$. For this, remember that the moduli functor for stable vector bundles over $X$
$$\mathcal{M}_X^{s}(n,d): (Sch/\Spec(\F))^{op}\rightarrow Sets$$
is representable if and only if $n$ and $d$ are coprime (see \cite{Ra}, \cite{MFK}). So if $\mathcal{M}_X^{s}(n,d)$ is representable, then there exists a scheme $M_X^{s}(n,d)$, which represents the moduli functor $\mathcal{M}^{s}_X(n,d)$ and therefore we get also a
universal family $\mathcal{U}$ of stable vector bundles parametrised by the scheme $M^{s}_X(n,d)$.
Now, if the pair of integers $(k,l)$ satisfies (\ref{ec1}) and (\ref{ec2}), then as we saw before $(k,l)$-stability implies stability.
Moreover, as $(k,l)$-stability is an open condition, there exists a non-empty open subscheme
$M^{k,l}_X(n,d)\subset M^{s}_X(n,d)$ which represents the moduli functor $\mathcal{M}_X^{k,l}(n,d)$ and the restriction of the universal family for stable bundles $\mathcal{U}|_{M_X^{k,l}(n,d)}$ to this subscheme is a universal family for $(k, l)$-stable bundles.  On the other hand, if $n$ and $d$ are not coprime, then  $\mathcal{M}_X^{s}(n,d)$ is universally corepresentable by a scheme $M_X^{s}(n,d)$ (see \cite[Definition 2.2.1]{HL} and \cite[Theorem 4.3.4]{HL}). Therefore the open subscheme $M_X^{k,l}(n,d)$ corepresents the moduli functor $\mathcal{M}_X^{k,l}(n,d)$.

In contrast, considering now the second case, where for the pair $(k, l)$ of integers we have
$k(n-1)+l<0$ or $k+l(n-1)<0$, then there exist semistable vector bundles which are
$(k,l)$-stable. Moreover, if $k$ and $l$ happen to be negative enough, then there are in fact unstable vector bundles
which are $(k,l)$-stable. This follows because for any vector bundle $E$
the slopes of its subbundles are always bounded above \cite[Lemma 2]{S}. Let us give two concrete examples to illustrate this.

\begin{exa}\label{exa-unssemi}
Let $E$ be an unstable vector bundle over $X$ of rank $rk(E)=2$ and degree $deg(E)=d$.
Let $L\subset E$ be a line subbundle of maximal degree. As
$deg(L)$ is bounded above, it follows that
$
deg(E)-2deg(L)
$
is bounded below. Hence if the pair $(k,l)$ is such that $deg(E)-2deg(L)>k+l,$ then
$E$ is $(k,l)$-stable.

\end{exa}
\begin{exa}
  Consider an unstable vector bundle $E$ over $X$ of rank $rk(E)=3$.
  Let $F\subset E$ be a subbundle of rank $rk(F)=2$ and maximal degree
  and let $L\subset E$ be a line subbundle of maximal degree.
  Suppose that the pair $(k,l)$ satisfies $2\,deg(E)-3\,deg(F)>k+2l$ and $deg(E)+3\,deg(L)>2k+l$.
  Then $E$ is an unstable and $(k,l)$-stable vector bundle.
\end{exa}

These last two examples can be extended to any rank, because if $E$ is a vector bundle
of rank $n$ then if $k$ and $l$ are negative enough, there exist unstable vector bundles of rank $n$ and degree $d$, which are $(k,l)$-stable.
For this just take an unstable vector bundle $E$ of rank $rk(E)=n$, a line subbundle $L\subset E$
and a subbundle $F\subset E$ of rank $rk(F)=n-1$ with $(n-1)\,deg(E)-n\,deg(F)>k+(n-1)l$ and $deg(E)+n\,deg(L)>(n-1)k+l$.
Therefore it follows that there are unstable vector bundles of rank $n$
and degree $d$ over an algebraic curve $X$, which are $(k,l)$-stable as soon as the integers $k$ and $l$ are negative enough.

Moreover, if $k(n-1)+l<0$ or $k+l(n-1)<0$ and the integers $n$ and $d$ are not coprime, then
the functor $\mathcal{M}_{X}^{k,l}(n,d)$ is not corepresentable. The reason for this is
that under these conditions there do exist semistable vector bundle of rank $n$ and degree $d$
which are also $(k,l)$-stable.

\begin{prop}\label{noncorep1}
   If $(k,l)$ is a pair of integers such that $k(n-1)+l<0$ or $k+l(n-1)<0$ and the integers $n$ and $d$ are not coprime
   then the moduli functor $\mathcal{M}_X^{k, l}(n, d): (Sch/\Spec(\F))^{op}\rightarrow Sets$ is not corepresentable.
\end{prop}

\begin{proof}
  Suppose that the pair of integers $(k,l)$ is such that $k(n-1)+l<0$. Let $E$ be a strictly semistable and indecomposable vector bundle.
  Furthermore, assume $E$ is such that the Jordan-H\"older filtration of $E$ is $0\subset L\subset E$.
 Then the associated graded $Gr(E)$ of $E$ is given as $Gr(E)=L\oplus F$ and $E\not\simeq L\oplus F$, with $F=E/L$.
   Then $E$ is $(k,l)$-stable. Moreover, we can construct a family $\mathcal{E}\rightarrow X\times\mathbb{A}^{1}$
   such that $\mathcal{E}|_{X\times\{0\}}=L\oplus F$ and $\mathcal{E}|_{X\times\{t\}}=E$ with $t\neq 0$ (see \cite[Lemma 16]{Se}).
   This gives rise to a jump phenomenon and determines the non-corepresentability of the moduli functor.
   In the case that $k+l(n-1)<0$, consider a strictly semistable and indecomposable  vector bundle $E'$,
   such that its Jordan-H\"older filtration is equal to $0\subset F' \subset E'$, where $F'$ is
   a rank $n-1$ subbundle of $E$. Hence $E'$ is a $(k,l)$-stable vector bundle such that
   $Gr(E')=F'\oplus L'$, where $L'=E/F$ is a line bundle. Moreover as before there exists a family
   $\mathcal{E'}\rightarrow X\times\mathbb{A}^{1}$
   such that $\mathcal{E'}|_{X\times\{0\}}=L'\oplus F'$ and $\mathcal{E'}|_{X\times\{t\}}=E'$ with $t\neq 0$.
    \end{proof}

We can now also give a description of the moduli spaces $M_X^{k,l}(n,d)$
of $(k,l)$-stable vector bundles of rank $n$ and degree $d$ over $X$ in terms of Geometric Invariant Theory,
always under the condition that the pair $(k, l)$ satisfies both inequalities (\ref{ec1}) and (\ref{ec2}).
This description will be needed later for comparison with
the respective moduli stacks. Recall that if the inequalities  (\ref{ec1}) and (\ref{ec2}) hold for a pair of integers $(k, l)$,
then $(k,l)$-stability implies stability and hence the moduli functor for $(k,l)$-stable vector bundles and its representability
by schemes follows in a natural way from the construction of the moduli space of stable bundles over $X$.
The construction of the moduli spaces $M_X^{k,l}(n,d)$ of $(k,l)$-stable vector bundles of rank $n$ and degree $d$ over $X$
is then a standard procedure using methods from Geometric Invariant Theory (see \cite{MFK}, \cite{HL}).
We will reproduce the construction here for the convenience of the reader as we will later need this explicit description
of the moduli spaces to compare them with the respective moduli stacks of $(k, l)$-stable vector bundles.

\begin{thm} Assume that the pair $(k, l)$ of integers satisfies the conditions that $0\leq k(n-1)+l < (n-1)(g-1)$
and $0\leq k+l(n-1)< (n-1)(g-1)$.
Then the moduli space $M_X^{k,l}(n,d)$ of $(k, l)$-stable vector bundles of rank $n$ and degree $d$ over $X$ exists and is an open subscheme of the moduli space $M_X^{s}(n,d)$ of stable vector bundles of rank $n$ and degree $d$.
\end{thm}

\begin{proof}
Let $\mathcal{O}_{X}(1)$ be an ample line bundle over $X$.
There exist integers $t$ and  $N$ such that for any sheaf $E$ over $X$
of rank $n$ and degree $d$, $E(t):=E\otimes \mathcal{O}_{X}(t)$ is generated by sections
and $h^{0}(X,E(t))= N$. We define $V:=\mathcal{O}_{X}^{N}$ and $\mathcal{H}:=V\otimes \mathcal{O}_{X}(-t)$.
Thus, the surjection $\mathcal{H}\to E\to 0$ determines a closed point in the respective Quot-scheme
$Quot_{\mathcal{H}}^{n,d}$.

We now consider the open subscheme $R^{k,l}\subset Quot_{\mathcal{H}}^{n,d}$
given as follows:  The quotient sheaves $\mathcal{H}\to F\to 0$ parameterised
by $R^{k,l}$ are locally free, $(k,l)$-stable and such that $V=H^{0}(\mathcal{H}(t))\cong H^{0}(F(t))$.
The scheme $R^{k,l}$ therefore parametrises all $(k,l)$-stable vector bundles
together with a choice of a base for the vector space $H^{0}(X,E(t)).$

Hence $R^{k,l}$ parametrises all  $(k,l)$-stable vector bundles of rank $n$
and degree $d$ over $X$. The general linear group $GL(N)$ acts on $Quot_{\mathcal{H}}^{n,d}$ and $R^{k,l}$
is invariant under this action. Moreover, this action factors through $PGL(N)$.
Therefore, the moduli scheme of $(k,l)$-stable vector bundles exists and is given by the
GIT quotient $M_X^{k,l}(n,d)= R^{k,l}//PGL(N)$.
\end{proof}

\section{Geometry of the moduli spaces $A_X^{k,l}(n,d)$.}
In the last section we studied the moduli problem for $(k, l)$-stable vector bundles over an algebraic curve $X$ in the particular case when the pair $(k,l)$ meets the conditions (\ref{ec1}) and (\ref{ec2}). Now we will analyse what happens in the more general case when $(k,l)$ is {\it any} pair of integers.
As we mentioned in Example (\ref{exa-unssemi}) there exist $(k,l)$-stable vector bundles which are not necessarily stable.
Moreover, as we saw in the last section the moduli functor $\mathcal{M}^{k,l}_{X}(n,d)$ is not always even corepresentable.
For this reason we will later consider a more general approach to the classification problem using the language of algebraic stacks.
But before let us make the following general observations concerning $(k, l)$-stable vector bundles over an algebraic curve $X$ for {\it any} pair $(k, l)$ of integers.

Let $A_X^{k,l}(n,d)$ denote the set of isomorphism classes of $(k,l)$-stable vector bundles
of rank $n$ and degree $d$ over $X$ for any given pair $(k, l)$ of integers. By the definition of $(k,l)$-stable vector bundles we readily get the following filtrations of sets:
$$
\cdots\subset A_X^{k-4,l}(n,d)\subset A_X^{k-3,l}(n,d)\subset A_X^{k-2,l}(n,d)\subset A_X^{k-1,l}(n,d)\subset A_X^{k,l}(n,d)\subset\cdots
$$
$$
\cdots\subset A_X^{k,l-4}(n,d)\subset A_X^{k,l-3}(n,d)\subset A_X^{k,l-2}(n,d)\subset A_X^{k,l-1}(n,d)\subset A_X^{k,l}(n,d)\subset\cdots
$$

Furthermore, if $(k,l)=(0,0)$ then there is a bijection between $A_X^{0,0}(n,d)$ and the rational points of the scheme $M_X^{s}(n,d)$.
More generally, if $(k,l)$ meets the conditions $(\ref{ec1})$ and (\ref{ec2}), then $A_X^{k,l}(n,d)$ is in
bijection with the rational points of the moduli scheme $M^{k,l}_{X}(n,d)$ as defined in the last section.
This induces a geometrical structure on $A_X^{k,l}(n,d)$ making it into a scheme and in this case the geometry of the moduli space is given as discussed in
the last section.

Now we will in contrast discuss how to induce a geometric structure on the sets $A_X^{k,l}(n,d)$,  in the complementary cases, when the
inequalities $(\ref{ec1})$ and (\ref{ec2}) do not hold for the pair $(k, l)$ of integers.

By definition of $(k,l)$-stability, if the pair of integers $(k,l)$ does not satisfy
$0\leq k(n-1)+l$ or $0\leq k+l(n-1)$ (see conditions (\ref{ec1}) and (\ref{ec2})),
then any stable vector bundle over the algebraic curve $X$ is also $(k,l)$-stable i.e., as sets we have an inclusion
$M_X^{s}(n,d)\subset A_X^{k,l}(n,d)$.
However, if $k$ and $l$ are both negative enough, then there are semistable and unstable vector bundles
which are also $(k,l)$-stable. The following results present some of the structure that appears in these complementary cases.

\begin{lem}\label{lemmaE}
If $E$ is an element of $A_X^{-1,1}(n,nt)$, then $E$ is semistable.
\end{lem}

\begin{proof}
Suppose that there exists a subbundle $F\subset E$, such that $\mu(F)>\mu(E)$.
By the $(-1,1)$-stability of $E$ we have that $\mu_{-2}(E)>\mu_{-1}(F)$ and therefore
$$
0> t-\mu(F)>\frac{2}{n}-\frac{1}{m}.
$$
But this would imply
$
0>tm-d(F)>-1+\frac{2m}{n},
$
which is impossible.
\end{proof}

By Lemma \ref{lemmaE} we therefore have a map
$$
A_X^{-1,1}(n,nt)\to M_X^{ss}(n,nt),
$$
where $M_X^{ss}(n,nt)$ is the moduli space of semistable vector bundles over $X$, and this map
sends the isomorphism class of $E$ at its $S$-equivalence class.

\begin{exa}
Let $t\in \mathbb{Z}$ be any integer and let $E\in A_X^{-1,1}(3,3t)$ be a strictly semistable vector bundle.
Hence for any subbundle $F$ of rank 2, we have that
$$\mu(E)-\mu(F)>0.$$
Furthermore, the Harder-Narasimhan filtration of $E$ is simply given as $0\subset L\subset E$ and
the associated graded is $Gr(E)=(E/L)\oplus L$, which determines the
$S$-equivalence class.
\end{exa}

We will now discuss  in detail how to induce a geometric structure on the sets $A_X^{k,l}(n,d)$, when $(k,l)$ satisfies
the inequalities
\begin{equation}\label{ec3}
k(n-1)+l<0,
\end{equation}
\begin{equation}\label{ec4}
k+l(n-1)<0.
\end{equation}

Again by the definition of $(k,l)$-stability, if the pair of integers $(k,l)$ satisfies the conditions (\ref{ec3}) and (\ref{ec4}),
then any stable vector bundle over $X$ is also $(k,l)$-stable i.e., as sets we have an inclusion
$M_X^{s}(n,d)\subset A_X^{k,l}(n,d)$.
And if $k$ and $l$ are both negative enough, then there exists again semistable and unstable vector bundles
which are also $(k,l)$-stable as we have seen before.

Recall also that any morphism of vector bundles can be factorized by
a morphism of maximal rank (see \cite[\S4]{NS}), i.e., if $f:E\to F$ is a morphism of vector
bundles, then we have the following diagram
$$
\xymatrix@1{E\ar[d]_{f}\ar[r]& E_{1}\ar[d]^{g}\ar[r]&0\\
F& F_{1}\ar[l]&0,\ar[l],\\
}
$$
where $g$ is of maximal rank.  The subbundle $F_{1}$ of $F$ is called the {\it subbundle determined by the image of} $f$
and $rk(f)$ is defined as the rank $rk(F_1)$.

\begin{lem}\label{lemmaA}
Let $E,F$ be two $(k,l)$-stable vector bundles over $X$.
If $f:E\to F$ is a morphism of vector bundles,
then we have:
$$
\mu_{k-l}(E)<\mu_{k-l}(F)-\frac{k+l}{rk(f)}.
$$
\end{lem}
\begin{proof}
With the notations introduced above we readily see that
$$\mu_{k-l}(E)<\mu_{-l}(E_{1})\leq \mu_{-l}(F_{1})=\mu_{k}(F_{1})-(k+l)/rk(F_{1})<\mu_{k-l}(F)-(k+l)/rk(F_{1}),$$
which proves our assertion.
\end{proof}

From this it follows immediately:

\begin{cor}
Let $E,F$ be two $(k,l)$-stable vector bundles over $X$, which both have the same rank and degree.
If $k+l\geq 0$ and $Hom(E,F)\neq 0$, then $E\cong F$.
In particular, if $E$ is $(k,l)$-stable then $E$ is simple.
\end{cor}

\begin{lem}\label{lemmaB}
If $E$ is an element of $A_X^{k,l}(n,d)$ such that $d> 2n(g-1-l)+l-k$,
then $H^{1}(E)=0.$
\end{lem}

\begin{proof}
Suppose that $H^{1}(E)\neq 0$, then by Serre duality $H^{0}(E^{*}\otimes \omega_{X})\neq 0$.
Hence $h:E\to \omega_{X}$ gives $\mu_{k-l}(E)\leq \mu_{k-l}(\omega_{X})-k-l=2g-2-2l$, i.e.,
$\mu_{k-l}(E)\leq 2(g-1-l)$ and $d\leq 2n(g-1-l)+l-k$, which  is a contradiction.
\end{proof}

\begin{lem}
If $E\in A_X^{k,l}(n,d)$, such that  $d>(2g-2l-1)n+l-k$,
then $E$ is generated by sections.
\end{lem}
\begin{proof}
Consider the exact sequence
$$0\to E(-x)\to E\to E_{x}\to 0,$$
and the associated long exact sequence in cohomology
$$
0\to H^{0}(E(-x))\to H^{0}(E)\to H^{0}(E_{x})\to H^{1}(E(-x))\to \cdots
$$
and observe that  $deg(E(-x))=d-n.$

Now, if $H^{1}(E(-x))\neq 0$ then Lemma \ref{lemmaB} implies
$d-n\leq 2n(g-1-l)+l-k$, i.e., $d\leq (2g-2l-1)n+l-k$, which is a contradiction.
Therefore, $H^{1}(E(-x))=0$ and hence $E$ is generated by sections.
\end{proof}

\begin{lem}
Let $k\leq l$  and let $E$ be a $(k,l)$-stable vector bundle of $X$ of slope $\mu$.
Suppose that $F\subset E$ is a subbundle of slope $\mu$.
Then $E/F$ is $(k,l)$-stable and $\mu(E/F)=\mu$.
\end{lem}
\begin{proof}
Suppose that $E/F$ is not a $(k,l)$-stable vector bundle. Then there exists
a subbundle $G\subset E/F$ such that
$$
0\to G\to E/F\to H\to 0
$$
and $\mu_{k}(G)\geq\mu_{k-l}(E/F)\geq \mu_{-l}(H)$.
This implies
$$
\mu(E/F)+\frac{k-l}{n-m}\geq \mu_{-l}(H)>\mu_{k-l}(E)=\mu(E)+\frac{k-l}{n}=\mu(E/F)+\frac{k-l}{n},
$$
where $m$ is the rank of $F$. But this holds if and only if  $k-l>0$, which gives a contradiction.
\end{proof}

\section{Moduli Stacks of $(k,l)$-stable vector bundles over an algebraic curve}

We will now consider the moduli stack $\bndkl$ of $(k, l)$-stable vector bundles of rank $n$ and
degree $d$ over the algebraic curve $X$ for {\it any} pair $(k, l)$ of integers. We will show that this moduli stack is an Artin algebraic stack, which is locally of finite type, reduced and irreducible.
As $(k,l)$-stability is an open condition, the moduli stack $\bndkl$ will in fact be an open substack of
the moduli stack $\mathscr{B}un_X(n,d)$ of all vector bundles of rank $n$ and degree $d$ over $X$ and will govern a good part of its geometry.

For any pair of integers $(k,l)$, and any scheme $T$, we define the groupoid of sections $\bndkl(T)$ as follows:
An object $E$ of $\bndkl(T)$ is a flat family of $(k, l)$-stable vector bundles of rank $n$ and degree $d$ over $X$ parametrised by $T$. The morphisms of $\bndkl(T)$ are the isomorphisms of these families.
Equivalently, $E$ is an object of $\bndkl(T)$, if $E$ is a vector bundle over $X\times T$ such that for each point $t\in T$ the restriction $E_{t}$ is a $(k,l)$-stable vector bundle of rank $n$ and degree $d$ over $X$.

Observe that $(k,l)$-stability is a property, which is stable under arbitrary base change,
i.e., if $f:T'\to T$ is a morphism of schemes and $E$ is an object of the groupoid $\bndkl(T)$, then $f^{*}E$ is an object of the groupoid $\bndkl(T')$.
Hence we get a lax $2$-functor or pseudo-functor, i.e. a prestack of the form
$$
\mathscr{B}un_X^{k,l}(n,d):(Sch/\Spec(\F))^{op}\to \mathfrak{}{Gpds}
$$
from the category of schemes over $\Spec(\F)$ to the $2$-category of groupoids, which associates to each scheme $T$ the groupoid $\bndkl(T)$ and to each morphism of schemes $f: T'\rightarrow T$ the functor $f^*: \bndkl(T')\rightarrow \bndkl(T)$ induced by the pullback operation on vector bundles. In addition, we have a natural isomorphism between the pullback functors, i.e., for each two composable morphisms $T''\stackrel{g}\rightarrow T'\stackrel{f}\rightarrow T$ we have a natural isomorphism between the functors $\epsilon_{f, g}: g^*\circ f^*\cong (f\circ g)^*.$

It follows that the necessary descent conditions hold with respect to the \'etale topology on $Sch/\Spec(\F)$ and therefore $\bndkl$ is a stack (see \cite[Exposé VIII, Thm 1.1, Prop. 1.10]{Gro}). In fact it is an Artin algebraic stack, which is an open substack of the moduli stack
$\mathscr{B}un_X(n,d)$ of all rank $n$ and degree $d$ vector bundles over $X$ as our main theorem shows:

\begin{thm}\label{thm-stack}
The moduli stack $\bndkl$ of $(k,l)$-stable vector bundles of rank $n$ and degree $d$
over an algebraic curve $X$ is a smooth Artin algebraic stack, which is locally of finite type. Moreover, the forgetful morphism
$\theta^{k,l}:\bndkl \to \mathscr{B}un_X(n,d)$ is a representable open embedding.
\end{thm}

This theorem will be a consequence of a more general result stated below.
For this we will need the following definition:

\begin{defn}\label{open}
A property $\mathfrak{P}$ of vector bundles over $X$ is an \textit{open property},
if for any family of vector bundles over $X$ parameterised by a scheme $T$
the set $T^{\mathfrak{P}}:=\{t\in T | \, E_{X\times\{t\}} \text{\ has property $\mathfrak{P}$}\}$ is a Zariski open subset of $T$.
\end{defn}

Given an open property $\mathfrak{P}$ of vector bundles over $X$ we can define again a prestack of the form
$$
\bp: (Sch/\Spec(\F))^{op}\to \mathfrak{}{Gpds}
$$
from the category of schemes over $\Spec(\F)$ to the $2$-category of groupoids, which associates to each scheme $T$ the groupoid $\bp(T)$ of families of vector bundles of rank $n$ and degree $d$ over $X$ having property $\mathfrak{P}$ and to each morphism of schemes $f: T'\rightarrow T$ the functor $f^*: \bp(T')\rightarrow \bp(T)$ induced via pullbacks.

The following fundamental theorem shows that this prestack $\bp$ is in fact an algebraic stack, the {\it moduli stack of vector bundles of rank $n$ and degree $d$ over $X$ having property} $\mathfrak{P}$.

\begin{thm}\label{prop-gen}
Let $\mathfrak{P}$ be an open property of vector bundles of rank $n$ and degree $d$ over an algebraic curve $X$.
Then the following holds:
  \begin{enumerate}
    \item The prestack $\bp$  defined by $\mathfrak{P}$ is a substack
          of the moduli stack $\mathscr{B}un_X(n,d)$.
    \item The forgetful morphism f: $\bp\to \mathscr{B}un_X(n,d)$ is representable by schemes.
    \item The moduli stack $\bp$ is an open algebraic substack of the moduli stack $\bnd$.
    \end{enumerate}
\end{thm}

\begin{proof}
(1) This follows again from the descent properties \cite[Exposé VIII, Théor\`eme 1.1, Proposition 1.10]{Gro}.

(2) Consider a scheme $Y$ and a morphism of stacks $g:{Y}\to \mathscr{B}un_X(n,d)$.
By the $2$-Yoneda lemma, $g$ corresponds to a family $E\to X\times Y$ and we have the following $2$-cartesian diagram

\begin{equation}\label{fibpro}
  \xymatrix@1{{Y}\times_{\mathscr{B}un_X(n,d)}\bp\ar[r]^{\hspace{1cm}p_1}\ar[d]_{p_2}&{Y}\ar[d]^{g}\\
  \bp\ar[r]_{f}&\mathscr{B}un_X(n,d).}
\end{equation}

We denote by $Y^{\mathfrak{P}}$ the Zariski open subset in $Y$ defined as
$$
Y^{\mathfrak{P}}:=\{y\in Y |\ E_{X\times\{y\}} \mbox{ has property $\mathfrak{P}$} \}.
$$
Now for any scheme $T$ over $\Spec(\F)$, the groupoid $(\bp\times_{\bnd} {Y})(T)$ is defined as follows:
\begin{enumerate}
   \item[i)] Elements of $Obj(\bp\times_{\bnd} {Y})(T)$, are triples $(\beta,F,\psi)$ such that
         $\beta:T\to Y$ is a morphism of schemes,
         $F\to X\times T$ is a family of  vector bundles of rank $n$ and degree $d$ over $X$ having property $\mathfrak{P}$ and
         $\psi:F\to(id_{X}\times\beta)^{*}E$ is an isomorphism of vector bundles.
         Observe that the existence of $\psi$ implies that $\beta$ factorizes through $Y^{\mathfrak{P}}$, i.e.,
         $\beta: T\to Y^{\mathfrak{P}}\hookrightarrow Y$.
   \item[ii)] Elements of $Mor (\bp\times_{\bnd} {Y})(T)$ are by definition given as pairs
         $(\alpha, \alpha'):(\beta,F,\psi)\to (\beta',F',\psi')$ such that
         $\alpha:\beta\to \beta'$ and $\alpha':F\to F'$ are morphisms in
         ${Y}(T)$ and $\bnd(T)$ respectively and such that the following diagram commutes:
\end{enumerate}

\begin{equation}
\xymatrix@1{f(F)\ar[r]^{f(\alpha')}\ar[d]_{\psi}&f(F')\ar[d]^{\psi'}\\
g(\beta)\ar[r]_{g(\alpha)}&g(\beta').}
\end{equation}

However, $\alpha$ is the identity map, hence $\beta=\beta'$, $g(\alpha)=id_{g(\beta)}$ and
$g(\beta')=g(\beta)=(id_{X}\times \beta)^{*}E$.
Moreover, $f(\alpha')=\alpha'=(\psi')^{-1}\circ\psi.$
Therefore, the morphisms are $(id_{\beta},(\psi')^{-1}\circ\psi):(\beta,F,\psi)\to (\beta,F',\psi')$
and this implies that the objects in the groupoid $(\bp\times_{\bnd} {Y})(T)$
do not have non-trivial automorphisms.

(3) The stack ${Y}^{\mathfrak{P}}$ is isomorphic
to the fiber product of stacks ${Y}\times_{\mathscr{B}un_X(n,d)}\bp$. Hence by (2) the claim follows.
\end{proof}

Now similar as in the proof of the algebraicity for the moduli stack $\bnd$ of vector bundles of rank $n$
and degree $d$ over $X$ (see for example \cite[Theorem 2.67]{Fra}, \cite[Proposition]{Tom}), we get the following:

\begin{lem}\label{lemA}
The diagonal $\Delta$ of $\bp$ is representable by a scheme,
quasi-compact and separated.
 \end{lem}

\begin{proof}
 Let $T$ and $T'$ be two schemes and let $E\to X\times T$, $E'\to X\times T'$ be
 two families of vector bundles. Then we have the following $2$-cartesian diagram:

\begin{equation}\label{clos-subcat}
\xymatrix@1{Isom(T\times T', pr_{1}E, pr_2 E')\ar[r]\ar[d]_{h}&\bp\ar[d]^{\Delta}\\
T\times T'\ar[r]_{\hspace*{-1.3cm}(E,E')}&\bp\times\bp .}
\end{equation}

And it follows that the sheaf $Isom(T\times T', pr_{1}E,pr_2 E')$ is a subscheme
of the fiber bundle $Hom(pr_{1}E, pr_2 E')$ on $T\times T'$.
Moreover, the morphism
$$h:Hom(pr_{1}E, pr_2 E')\to T\times T'$$
is affine and therefore the result follows.
\end{proof}

Finally from this we now get the desired result:

\begin{thm}\label{thm-stackp}
The moduli stack $\bp$ is a smooth Artin algebraic stack, which is locally of finite type.
\end{thm}

\begin{proof}

Consider an atlas ${U}$ of $\bnd$ and a smooth surjective morphism ${U}\to \bnd$.
By (3) of Theorem \ref{prop-gen}, the $2$-fiber product ${U} \times_{\bnd} \bp$ is representable by a scheme.
Now we will prove that ${U}^{\mathfrak{P}}:={U}\times_{\bnd}\bp$ is an atlas and
${U}^{\mathfrak{P}} \to \bp$ is representable, smooth and locally of finite type.
For this, we consider a scheme $T$, a morphism $h:T\to \bp$ and the following diagram:
$$
\xymatrix@1{&{U}^{\mathfrak{P}} \ar[dd]\ar[rrr]&&&{U}\ar[dd]\\
{T}\times_{\bp}{U}^{\mathfrak{P}}\ar[dd]\ar[ru]\ar[rrrru] &&&&\\
&\bp\ar[rrr]^{f}&&&\bnd\\
{T}\ar[ru]^{h}\ar[rrrru]_{f\circ h}&&&\\
}
$$

Hence ${T}\times_{\bp}{U^{\mathfrak{P}}}\to{T}$ is smooth and locally of finite type
because the atlas ${U}\to \bnd$ is. This implies that ${U}^{\mathfrak{P}} \to \bp$
is representable, smooth and locally of finite type.
The quasi-separedness of the diagonal $\Delta$ of $\bp$ is a consequence of Lemma \ref{lemA}.
\end{proof}

\begin{proof}[Proof of Theorem \ref{thm-stack}]
Observe that the proof is simply a direct consequence of Theorem \ref{thm-stackp}
and (2) in Theorem \ref{prop-gen} as $(k, l)$-stability is an open property.
\end{proof}

The above considerations also imply immediately the following:

\begin{cor}
For any pair $(k, l)$ of integers, the moduli stacks $\mathscr{B}un^{k-1,l}_X(n,d)$ and $\mathscr{B}un^{k,l-1}_X(n,d)$
are open substacks of the moduli stack $\mathscr{B}un^{k,l}_X(n,d)$.
\end{cor}

Thus, we get a filtration of the moduli stack $\mathscr{B}un_X(n,d)$ of all vector bundles over $X$ of rank $n$ and degree $d$ by open substacks in the following way:
$$
\cdots\subset \mathscr{B}un_X^{k-3,l}(n,d)\subset \mathscr{B}un_X^{k-2,l}(n,d)\subset \mathscr{B}un_X^{k-1,l}(n,d)\subset \mathscr{B}un_X^{k,l}(n,d)\subset\cdots
$$
$$
\cdots\subset \mathscr{B}un_X^{k,l-3}(n,d)\subset \mathscr{B}un_X^{k,l-2}(n,d)\subset \mathscr{B}un_X^{k,l-1}(n,d)\subset \mathscr{B}un_X^{k,l}(n,d)\subset\cdots
$$
We will now give also an explicit construction of an atlas for the moduli stack $\bndkl$, which
will be used later and is of interest in its own right.

The diagonal $\Delta$ of $\bndkl$ is quasi-compact by Lemma \ref{lemA}. Hence it is enough to prove that $\bndkl$ has a smooth atlas
in order to prove the smoothness of $\bndkl$. For this, consider the following explicit construction of an atlas.
For vector bundles of rank $n$ and degree $d$ consider the Hilbert polynomial $P_{n,d}(x):=nx+d+n(1-g)$ and denote by
$P(m)$ the number $P(m)=P_{n,d}(m)$ for a given integer $m$.

 Now let us consider the Quot scheme $Quot(\mathcal{O}^{P(m)}_X, P(x+m)).$

 For every integer $m$ we have an open subscheme $R^{k,l}_m$ given
 by the following conditions:
 \begin{enumerate}
 \item Every point in $R_{m}^{k,l}$ determines a quotient $(k,l)$-stable vector bundle of $\mathcal{O}^{P(m)}_X$.
 \item If $E$ is a family of quotients of $\mathcal{O}^{P(m)}_X$ parameterised by a scheme $T$,
        then $\mathcal{R}^{1}_{pr_{2}*}E=0$ and we have an
        isomorphism $\mathcal{O}^{P(m)}_X\cong \mathcal{R}^{0}_{pr_{2}*}E.$
 \end{enumerate}

With these conditions we see that the universal family $E^{univ}$
of  the Quot scheme determines a family parametrised by $R^{k,l}_{m}$ and therefore
a morphism
    $$
    r_{m}^{k,l}:R^{k,l}_{m}\to Quot(\mathcal{O}^{P(m)}_X, P(x+m)).
    $$

Take a point of $R_{m}^{k,l}$ represented by the exact sequence
$$0\to H\to \mathcal{O}_{X}^{P(m)}\to E\to 0$$
with $E$ a quotient $(k,l)$-stable vector bundle and $H$ the kernel.
Then we have that
$H^{1}(E\otimes H^{*})=0$, where $H^*$ is the dual vector bundle of $H$ and this implies that $r_m^{k, l}$ is smooth.
So we get that $r^{k,l}:=\coprod r^{k,l}_m$ is a smooth morphism.

Finally it follows from the previous constructions that $\bndkl$ is a smooth algebraic stack via a similar line of arguments as in \cite{Hei, Fra}.

\section{Gerbes and coarse moduli spaces of $(k, l)$-stable vector bundles.}

We can relate the moduli stacks of $(k, l)$-stable vector bundles $\bndkl$ and the moduli spaces $M_X^{k,l}(n,d)$, whenever the last ones exist. It turns out that they are actually coarse moduli spaces for the moduli stacks. This is very similar to the relation between moduli stacks and moduli spaces of stable bundles, which we will recall now in some details.
Let $\bs_X(n,d)$ be the moduli stack of stable vector bundles of rank $n$ and degree $d$ over $X$ and
$M_X^{s}(n,d)$ be the moduli space of stable vector bundles of rank $n$ and degree $d$ over $X$ as we discussed before.
By construction we have $\bs_X(n,d)=[R^{s}/GL_{N}]$ as a quotient stack and $M_X^{s}(n,d)=R^{s}//PGL_{N}$ as a GIT-quotient,
where $R^{s}$ is an open subscheme as defined in \cite{HL} (see also \cite{Tom}, \cite{Hei}).
There is also an asssociated morphism of stacks $\bs_X(n,d)\to M_X^{s}(n,d)$, such that all the fibers
are isomorphic to he classifying stack $\mathscr{B}\mathbb{G}_m$ of all line bundles. Here $\mathbb{G}_m$ is the multiplicative group over
$\Spec(\F)$, which in case we work over $\Spec(\C)$ is just $\C^*$. In fact more is true, the associated morphism is actually a gerbe (see \cite[Example 3.9]{Hei} and also \cite{LMB, Lieblich} for the general definition of a gerbe).

\begin{lem}\label{gerbe}
The morphism $\Phi: \bs_X(n,d)\to M_X^{s}(n,d)$ is a $\mathbb{G}_m$-gerbe.
\end{lem}

In addition, we have that $M_X^{s}(n,d)$ is a coarse moduli space for the algebraic stack $\bs_X(n,d)$ (see \cite{Tom, Hei}) and for the convenience of the reader we will present a proof here in order to obtain a similar result for the moduli stack of $(k, l)$-stable vector bundles.

\begin{prop}\label{coarse}
Let  $\Phi:\bs_X(n,d)\to M_X^{s}(n,d)$ be as before.
Then $M_X^{s}(n,d)$ is a coarse moduli space for the moduli stack  $\bs_X(n,d)$.
\end{prop}

\begin{proof}
To simplify the notation for this proof, we will write $\bs_{n,d}$ instead of $\bs_{X}(n,d)$. We will now prove that $M_X^{s}(n,d)$ is a coarse moduli space for the algebraic stack $\bs_{n,d}$.
First, observe that for any algebraically closed field $\F$, the morphism $\Phi_{\Spec(\F)}:\bs_{n,d}(\Spec(\F))\rightarrow M^{s}_{X}(n,d)(\Spec(\F))$ is a bijection using the definitions and following the same line of arguments as in \cite[Example 3.7]{Hei} (see also \cite{Fra}).  Let now $Y$ be any scheme and $\Psi:\bs_{n,d}\to {Y}$ be a morphism of stacks.
We will construct a morphism of stacks $\Theta:{M}_X^{s}(n,d)\to {Y}$, such that the following diagram commutes:
$$
\xymatrix@1{&\bs_{n,d}\ar[d]^{\Phi}\ar[dl]_{\Psi}\\
{Y}&{M}_X^{s}(n,d)\ar[l]^{\Theta\ \ \ }.
}
  $$
  So for any scheme $T$ we denote by $\Phi_{T}:\bs_{n,d}(T)\to M_X^{s}(n,d)(T)$ the corresponding
  morphism  of groupoids.

  Hence if $ T=M_X^{s}(n,d)$, there exists an object $\beta\in Obj(\bs_{n,d}(M_X^{s}(n,d))$, such that
  $\Phi_{M_X^{s}(n,d)}(\beta)=Id_{M_X^{s}(n,d)}$.
  We let $\Theta:M_X^{s}(n,d)\to Y$ be the morphism obtained as the image of $\beta$ under $\Psi$, i.e., we set $\Theta:=\Psi(\beta)$.
  We will prove that $\Theta$ does not depend of the choice of $\beta$.
  Suppose that $\beta$ and $\beta'$ are such that $\Phi_{M_X^{s}(n,d)}(\beta)=\Phi_{M_X^{s}(n,d)}(\beta')=Id_{M_X^{s}(n,d)}.$
  By Lemma \ref{gerbe}, $\beta$ and $\beta'$ are locally isomorphic, i.e.,
  there exist a line bundle $L$ over $M_X^{s}(n,d)$ such that
  $\beta\cong\beta'\otimes p_{M_X^{s}(n,d)}^{*}L$. Thus, there is a cover $\{U_{i}\}_{i\in I}$
  of $M_X^{s}(n,d)$ with the following property for all $i\in I$:
  $$
  \beta|_{U_{i}}\cong\beta'|_{U_{i}}.
  $$
  Hence for all $i\in I$ we get now:
  $$
  \Psi_{M_X^{s}(n,d)}(\beta)|_{U_{i}}= \Psi_{M_X^{s}(n,d)}(\beta|_{U_{i}})=\Psi_{M_X^{s}(n,d)}(\beta'|_{U_{i}})=\Psi_{M_X^{s}(n,d)}(\beta')|_{U_{i}}.
  $$
Therefore, $\Psi_{M_X^{s}(n,d)}(\beta)=\Psi_{M_X^{s}(n,d)}(\beta')$ and this proves independence of the choice of $\beta$.

Now for any scheme $T$, we need to prove commutativity of the diagram
$$
\xymatrix@1{&\bs_{n,d}(T)\ar[d]^{\Phi_{T}}\ar[dl]_{\Psi_{T}}\\
{Y}(T)&{M}_X^{s}(n,d)(T)\ar[l]^{\Theta_{T}\ \ \ }.
}
$$
To do this, we consider $\alpha\in \bs_{n,d}(T)$ and $\Phi_{T}(\alpha)\in M_X^{s}(n,d)(T)$
which determines a morphism  $\Phi_{T}(\alpha):T\rightarrow M_X^{s}(n,d)$. Hence we have the diagram
$$
\xymatrix@1{&\bs_{n,d}(M_X^{s}(n,d))\ar[dd]^{\Phi_{M_X^{s}(n,d)}}\ar[dddl]_{\Psi_{M^{s}_{X}(n,d)}}\ar[rrr]^{\bs_{n,d}(\Phi_{T}(\alpha))}
&&&\bs_{n,d}(T)\ar[dd]^{\Phi_{T}}\ar[dddl]_{\Psi_{T}}\\
&&&&\\
&{M}_X^{s}(n,d)(M_X^{s}(n,d))\ar[ld]^{\Theta_{M_X^{s}(n,d)}}\ar[rrr]^{{M}_X^{s}(n,d)(\Phi_{T}(\alpha))}&&&
{M}_X^{s}(n,d)(T)\ar[ld]^{\Theta_{T}}\\
{Y}(M_X^{s}(n,d))\ar[rrr]^{{Y}(\Phi_{T}(\alpha))}&&&{Y}(T)
}
$$
where every square commutes.

Furthermore, there is a $\beta\in \bs_{n,d}(M_{X}^{s}(n,d))$ such that $\Phi_{M}(\beta)=id_{M^{s}_{X}(n,d)}$ and
$\bs_{n,d}(\Phi_{T}(\alpha))(\beta)=\alpha.$

Now, as $\Psi_{M^{s}_{X}(n,d)}(\beta)=(\Theta_{M^{s}_{X}(n,d)}\circ \Phi_{M^{s}_{X}(n,d)})(\beta)$, commutativity implies that
$\Psi_{T}(\beta)=(\Theta_{T}\circ \Phi_{T})(\beta)$.


Finally, suppose that there exists a morphism $\Gamma: M_X^{s}(n,d)\to Y$ such that
the following diagram commutes:
$$
\xymatrix@1{&\bs_X(n,d)\ar[d]^{\Phi}\ar[dl]_{\Psi}\\
{Y}&{M}_X^{s}(n,d)\ar[l]^{\Gamma \ \ \ }.\\
}
$$
\noindent Then, $\Gamma=\Gamma_{M_X^{s}(n,d)}(Id_{M_X^{s}(n,d)})=\Gamma\circ Id_{M_X^{s}(n,d)}=\Psi_{M_X^{s}(n,d)}(\beta)=\Theta\circ Id_{M_X^{s}(n,d)}=\Theta$, which finishes the proof.
\end{proof}

Now we get as consequences of the constructions of moduli spaces of $(k, l)$-stable vector bundles the following:

\begin{cor}
Let $(k,l)$ be a pair of integers satisfying the conditions (\ref{ec1}) and (\ref{ec2}).
Then the morphism $\bndkl \to M_X^{k,l}(n,d)$ is a $\mathbb{G}_m$-gerbe and $M_X^{k,l}(n,d)$ is
a coarse moduli space for the moduli stack $\bndkl$.
\end{cor}

\begin{proof}
This follows because $R^{k,l}$ as constructed before is a subscheme of $R^{s}$ and then by using Lemma \ref{gerbe} and Proposition \ref{coarse}.
\end{proof}

Following \cite[Cor. 3.12]{Hei} or \cite{Ra, DreNa} we can now reason as follows: Suppose that we are in the special case that $\gcd(n,d)=1$, then there exists a universal family, a Poincar\'e family $\mathcal{U}$ over $X\times M_X^{s}(n,d)$.
Moreover, if the pair $(k,l)$ satisfies the conditions (\ref{ec1}) and (\ref{ec2}), then $M_X^{k,l}(n,d)$
is an open subscheme of $M_X^{s}(n,d)$ and the restriction $\mathcal{U}|_{X\times M_X^{k,l}(n,d)}$
is the universal family over $X\times M_X^{k,l}(n,d)$.
Thus, the splitting of the gerbe $\bs_X(n,d)\to M_X^{s}(n,d)$ implies the splitting of
the gerbe $\bndkl\to M_X^{k,l}(n,d)$. On the other hand, for $\gcd(n,d)\neq 1$ it is well known that there is no open subset
$A\subset M_X^{s}(n,d)$, such that there exists a Poincar\'e family over $X\times A$.
Hence we have in this case that the gerbe $\bndkl\to M_X^{k,l}(n,d)$ does not split. So summarising we have shown:

\begin{cor}
Let $(k,l)$ be a pair of integers satisfying the conditions (\ref{ec1}) and (\ref{ec2}).
Then the $\mathbb{G}_m$-gerbe $\bndkl \to M_X^{k,l}(n,d)$ splits if and only if
$\bs_X(n,d)\to M_X^{s}(n,d)$ splits.
\end{cor}

Let us finally also recall the following relations between moduli stacks and coarse moduli spaces of stable bundles over $X$ (see \cite[Prop. 3.3]{TG}):

\begin{prop}
Let $\bs_X(n, d)$ be the moduli stack of stable vector bundles of rank $n$ and degree $d$ over $X$.
There is a commutative diagram of stacks

$$
\xymatrix@1{[R^{s}/GL(N)]\ar[r]^{q}\ar[d]_{g}^{\cong}&[R^{s}/PGL_{N}]\ar[d]^{h}_{\cong}\\
\bs_X(n, d)\ar[r]_{\varphi}&{M}_X^{s}(n,d).\\
}
$$
where $g$ and $h$ are isomorphisms of stacks.
\end{prop}

This now implies together with the above considerations readily the following relation between the moduli stacks and moduli spaces of $(k, l)$-stable vector bundles over $X$.

\begin{cor}\label{diagerb}
Let $(k,l)$ be a pair of integers satisfying the conditions (\ref{ec1}) and (\ref{ec2}). Let $\bndkl$ be the moduli stack of $(k,l)$-stable vector bundles of rank $n$ and degree $d$ over $X$. There is a commutative diagram of stacks

$$
\xymatrix@1{[R^{k,l}/GL(N)]\ar[r]^{q}\ar[d]_{g}^{\cong}&[R^{k,l}/PGL_{N}]\ar[d]^{h}_{\cong}\\
\bndkl\ar[r]_{\varphi}&{M}_X^{k,l}(n,d).\\
}
$$
where $g$ and $h$ are isomorphisms of stacks.
\end{cor}

\section{Cohomological properties of $\bndkl$ and $M_X^{k,l}(n,d)$.}

We will now derive some cohomological properties for the moduli stacks and moduli spaces of $(k, l)$-stable vector bundles over an algebraic curve.
Let us start with some general remarks on the cohomology of algebraic stacks.
Let $\mathscr{X}$ be an algebraic stack, which is smooth and locally of finite type over $\Spec(\F)$ where $\F$ is either the algebraic closure $\overline{\F}_q$ of the field $\F_q$ or the field $\C$ of complex numbers.

If $\F=\overline{\F}_q$, we use $l$-adic cohomology of the stack $\mathscr{X}$, where $l$ is a prime different from $p$. The $l$-adic cohomology of $\mathscr{X}$ is defined over the lisse-\'etale site $\mathscr{X}_{\text{lis-\'et}}$ of $\mathscr{X}$ and is given as the limit of the cohomologies of all the open substacks $\mathscr{U}$ of finite type of the given algebraic stack $\mathscr{X}$ (see \cite{HS}), i.e. we set
$$H^*( \mathscr{X}, \Ql)=\lim_{\substack{{\mathscr{U}\subset\mathscr{X}},\\
\text{open, finite type}}}\hspace*{-0.5cm}H^*(\mathscr{U}, \Ql).$$
If $\F=\C$, then we use rational cohomology $H^*( \mathscr{X}, \Q)$ of the stack $\mathscr{X}$ instead and all statements below hold if we replace $l$-adic cohomology everywhere with rational cohomology.

As a general reference for cohomology of algebraic stacks we refer to \cite{LMB} and especially for $l$-adic cohomology and its main properties to the general formalism of cohomology functors as developed by Behrend \cite{Be2, Be3}, and in subsequent work by Laszlo and Olsson \cite{LO1, LO2}. Concerning in particular the cohomology of the moduli stack $\bnd$ of all vector bundles of rank $n$ and degree $d$ over an algebraic curve $X$ we will also refer to \cite{HS, Fra, NSt}.

Let us assume throughout the rest of this section that the rank $n$ and degree $d$ of all of our vector bundles over the algebraic curve $X$ are coprime.
Then the moduli space of stable vector bundles $M_X^{s}(n,d)$ admits a universal family $E^{univ}$ of vector bundles of rank $n$ and degree $d$.
The existence of such a universal family means that the gerbe
$ \Phi: \mathscr{B}un^{s}_X(n,d)\to M_X^{s}(n,d)$ is neutral, (see \cite[Lemma 3.10]{Hei}) i.e., we have a splitting
$$
\mathscr{B}un^{s}_X(n,d)\cong {M}_X^{s}(n,d)\times \mathscr{B}\mathbb{G}_m
$$
where $\mathscr{B}\mathbb{G}_m$ is again the classifying stack of line bundles or principal $\mathbb{G}_m$- bundles.
Now using Corollary \ref{diagerb}, the restriction to $(k,l)$-stable vector bundles under the condition that for the pair $(k, l)$ of integers the inequalities $(\ref{ec1})$ and $(\ref{ec2})$ hold gives a gerbe
$$
\Phi^{k,l}:\mathscr{B}un_X^{k,l}(n,d)\to {M}_X^{k,l}(n,d).
$$
Moreover, we have that it splits, i.e.
$$
\mathscr{B}un_X^{k,l}(n,d)\cong {M}_X^{k,l}(n,d)\times \mathscr{B}\mathbb{G}_m.
$$
Hence, in this particular situation the cohomology of the moduli stack of $(k, l)$-stable vector bundles can be calculated directly as follows:
\begin{prop}
If the pair $(k,l)$ of integers satisfies the conditions $(\ref{ec1})$ and $(\ref{ec2})$, then
$$
H^{*}(\mathscr{B}un_X^{k,l}(n,d), \Ql)\cong H^{*}({M}_X^{k,l}(n,d), \Ql)\otimes H^{*}(\mathscr{B}\mathbb{G}_m, \Ql).
$$
\end{prop}

However, in the rank two case, it is possible to compute the cohomology of the moduli stack of $(k,l)$-stable
vector bundles over $X$ using the Semi-Purity Lemma  (see \cite[Lemma 2.2.2]{HS}). For this it is necessary to compute the codimension with respect
to the moduli stack of stable vector bundles.
\begin{lem}\label{lem-coh}
Let $(k,l)$ be a pair of integers such that $0\leq k+l <g-1$, then we have the following statements:
\begin{enumerate}
  \item We have
 $$
codim (\mathscr{B}un_X^{s}(2,d)\setminus \mathscr{B}un_X^{k,l}(2,d))= g-k-l-1.
$$
\item If $k=l,$ and $0\leq 3k \leq 2g-2,$ then
$$
codim (\mathscr{B}un_X^{s}(3,d)\setminus \mathscr{B}un_X^{k,l}(3,d))= 2g-3k-3.
$$
\end{enumerate}
\end{lem}
\begin{proof}
This Lemma is a direct consequence of Corollary \ref{diagerb} and the computation of the codimension of the moduli
space of $(k,l)$-stable vector bundles with respect to the moduli space
of stable bundles as given in Proposition \ref{codimprop}.
\end{proof}

Using the Semi-Purity Lemma we then get:
\begin{cor}\label{cor-coh}
Let $(k,l)$ be a pair of integers such that $0\leq k+l <g-1$, then we have
$$
H^{*}(\mathscr{B}un_X^{k,l}(2,d), \Ql)\cong H^{*}(\mathscr{B}un_X^{s}(2,d), \Ql)
$$
for $*<2(g-k-l-1)$ if d is odd.
\end{cor}

It is well known that if $\mathfrak{Z}\to\mathfrak{X}$ is an embedding of algebraic stacks
of codimension $c$, then by the associated Gysin sequence in cohomology we have the following isomorphism:
$$
H^{i}(\mathfrak{Z}, \Ql)\cong H^{i}(\mathfrak{Z}\setminus\mathfrak{X}, \Ql),
$$
whenever $i< 2c-1$ (see \cite[Lemma 2.2.2]{HS}).

Now for rank two vector bundles of even degree over $X$, the filtration given by $(k,l)$-stability can be rewritten as:
$$
\cdots\supseteq  \mathscr{B}un_X^{-2}(2,d)\supseteq \mathscr{B}un_X^{-1}(2,d)\supseteq \mathscr{B}un_X^{0}(2,d)\supseteq \mathscr{B}un_X^{1}(2,d)\supseteq \cdots
$$
The moduli stack $\mathscr{B}un_X^{-1}(2,d)$ corresponds to the moduli stack of semistable vector bundles and $\mathscr{B}un_X^{0}(2,d)$ corresponds to the moduli stack of stable vector bundles.
Hence $\mathscr{B}un_X^{stss}(2, d):=\mathscr{B}un_X^{-1}(2,d)\setminus \mathscr{B}un_X^{0}(2,d)$ is given by the {\it strictly} semistable vector bundles and determines a closed substack of $\mathscr{B}un_X^{-1}(2,d)$ and we get as a consequence:
\begin{cor}
There is an isomorphism
$$H^{i}(\mathscr{B}un_X^{stss}(2,d), \Ql)\cong H^{i}(\mathscr{B}un_X^{s}(2,d), \Ql)$$
for all $i<2\, codim(\mathscr{B}un_X^{stss}(2,d))-1.$
\end{cor}

We now describe some cohomological properties of the moduli stack  $\bs_{X}(2,d)$ using the Shatz polygon associated to vector bundles \cite{HL, S} and the Harder-Narasimhan filtration.

Consider a family $E\to X\times T$ of vector bundles of rank $2$
and degree $d$ and denote by $E_{t}=E|_{X\times\{t\}}$ the corresponding restriction.
Then, if $P=\{(0,0),(1,d_{1}),(2,d) \}$ denotes a Shatz polygon (see \cite{S, AD}), we define the following sets
$$
F_P(T)=\{t\in T | P(E_{t})> P\},
$$
$$
\Omega_P(T)=T\setminus F_P(T),
$$
$$
S_P(T)=\{t\in T|P(E_{t})=P\}.
$$

We have the following general description:

\begin{lem}\label{lastlem}
Let $(k,l)$ a pair of integers and $E$ be an unstable vector bundle of rank 2 and degree $d$.
Then the following statements are equivalent:
\begin{enumerate}
\item $E$ is $(k,l)$-stable.
\item If $P(E)=\{(0,0),(1,d_{1}),(2,d) \}$ is the Shatz polygon of $E$ and
    $0\subset L_{0} \subset E$ is the Harder-Narasimhan filtration $HN(E)$ of $E$, then  $2d_{1}=2d(L_{0})< d-k-l$.
\end{enumerate}
\end{lem}

\begin{proof}
  $E$ is $(k,l)$-stable if and only if $\mu(E)-\mu(L)>(k+l)/2$ for any subbundle $L\subseteq E$,
  which is equivalent to having $\mu(E)-\mu(L)\geq \mu(E)-\mu(L_{0})>(k+l)/2$ where  $L_0$ is the maximal subbundle of $E$.
\end{proof}

With $E$ and $(k,l)$ as in Lemma \ref{lastlem}, we see that $k+l\leq 0$.
Now we consider the Shatz polygon
$$
P_{k,l}:=\{(0,0),(1,d_{1}^{k,l}), (2,d) \},
$$
with $d_{1}^{k,l}$ defined as the biggest integer such that $2d_{1}^{k,l}< d-k-l$.

Then we have the following consequence:

\begin{prop}
Let $E\to X\times T$ be a family of vector bundles parametrised by a scheme $T$, then $E_t$ is $(k,l)$-stable
if and only if  $t\in \Omega_{P_{k,l}}(T)$.
\end{prop}

\begin{proof}
If $E_{t}$ is $(k,l)$-stable and the Harder-Narasimhan filtration $HN(E_{t})$ is $0\subset L_{0}\subset E$,
then by Lemma \ref{lastlem} we have:  $2d(L_{0})<d-k-l$, which implies that for the Shatz polygon we have: $P(E)<P_{k,l}$.
The converse follows in a similar way.
\end{proof}

As a nice direct consequence we also get:

\begin{cor}\label{cod-com}
If $E$ is a complete family of vector bundles and $T$ is a smooth scheme,
then $codim(T\setminus T^{k,l})=2 (2d_{1}^{k,l}-d+g-1)$
\end{cor}

\begin{proof}
As a first step we observe that $T\setminus T^{k,l}=S_{P_{k,l}}(T)$.
Hence we can apply \cite[ Corollary 15.4.3]{LP} and the result follows.
\end{proof}

From the above considerations we get a kind of approximation of the cohomology of the moduli stack of all rank two vector bundles over the algebraic curve by the cohomologies of the different moduli stacks of $(k, l)$-stable bundles of rank two, namely we have:

\begin{thm}
 $ \lim\limits_{\leftarrow}H^{i}(\mathscr{B}un_X^{k,l}(2,d), \Ql)=H^{i}(\mathscr{B}un_X(2,d), \Ql).$
\end{thm}
\begin{proof}
We will prove that $\mathscr{B}un_X^{k,l}(n, d)\to \mathscr{B}un_X(n, d)$ is an isomorphism in cohomology of degree $i$ if
$i< -2(k+l-3)$ and $k+l\equiv d\mod 2$ or $i< -2(k+l-2)$ and $k+l\not\equiv d\mod 2$.
However if $E\to X\times T$ is a complete family, then by Corollary \ref{cod-com}
the inclusion $T^{k,l}\to T$ is an isomorphism in cohomology of degree $i$ as above.
Then the results follows again by a Gysin sequence argument.
\end{proof}

With the results described above we can now also define a general Hecke correspondence for the moduli stacks $\mathscr{B}un_X^{k,l}(n,d)$
of $(k, l)$-stable vector bundles. Hecke correspondence have been defined and used in many contexts (see, \cite{Bra-Mat,TG,YH,LOZ}).
In particular, Hoffmann in \cite{NH} described a Hecke correspondence for the moduli stack of all vector bundles over an algebraic curve
using the evaluation map transformation and he constructed a vector bundle over any given open substack.

Consider the universal family $E^{univ}$ of vector bundles over $X\times \mathscr{B}un_X(n,d)$ and denote by
 $E^{univ}_{k,l}$ the restriction of the universal bundle to the substack of $(k, l)$-stable vector bundles $\mathscr{B}un_X^{k,l}(n,d)$.
 Observe that $E^{univ}_{k,l}$ has weight $1$.
 Hence if $1\leq r\leq n-1$, we can associate to $E^{univ}_{k, l}$ the Grassmannian bundle $Gr_{r}(E^{univ}_{k, l})$, which also has
 weight $1$. Using Proposition 3.9 of \cite{NH} we therefore get the following:
 \begin{prop}\label{bio-gras}
For any two pairs of integers $(k_{1}, l_{1})$ and $(k_{2}, l_{2})$ satisfying conditions $(\ref{ec1})$ and $(\ref{ec2})$, there exists a birational linear map
 $$
 \xymatrix{Gr_{j}(E_{k_{1},l_{1}}^{univ})\ar@{-->}[r]^{\rho}&Gr_{j}(E_{k_{2},l_{2}}^{univ})\\
 }
 $$
 over the moduli stack $\mathscr{B}un_X(n,d).$ If in addition $j$ is divisible by $\gcd(n,d),$ then the Grassmannian bundle
$\rho:Gr_{j}(E_{k,l}^{univ})\to \mathscr{B}un_X(n, d)$ is birational linear.
 \end{prop}

Thus by Proposition \ref{bio-gras} and Proposition $A.6$ in \cite{NH}, we therefore obtain the following:
\begin{prop}\label{bio-mod}
Let $m$ be an integer with $1\leq m\leq n-1$, $(k_{2}, l_{2})$ be a pair of integers, such that $(k_{2},l_{2}-m)$
satisfies the conditions $(\ref{ec1})$ and $(\ref{ec2})$ and $(k_{1},l_{1})$ be a pair of integers, such that $(k_{1},l_{1})$-stability implies $(k_{2},l_{2})$-stability. Then we have the following diagram of moduli spaces:
$$
\xymatrix{ M_X^{k_{1},l_{1}}(n,d)\ar@{-->}[r]^{\Psi_{1} \ \ \ \ \ \ }\ar@{^(->}[d]& M_X^{k_{1},l_{1}-m}(n,d-m)\ar@{^(->}[d]\\
M_X^{k_{2},l_{2}}(n,d)\ar@{-->}[r]_{\Psi_{2} \ \ \ \ \ \ }& M_X^{k_{2},l_{2}-m}(n,d-m)
}
$$
where $\Psi_{i}$ is a birational linear map of schemes for each $i=1,2$.
\end{prop}

Proposition \ref{bio-mod} therefore determines the following diagram of algebraic stacks:
{\small
$$
\xymatrix{ &&Gr_{m}(E^{univ}_{k,l})\ar@/_/[dl]\ar@/^/[rd]&\\
& \mathscr{B}un_X^{k_{1},l_{1}}(n, d)\ar@{-->}[rr]\ar@{->}'[d][dd] && \mathscr{B}un_X^{k_{1},l_{1}-m}(n,d-m)\ar@{->}[dd]\\
M_X^{k_{1},l_{1}}(n,d)\ar@{<-}[ur]\ar@{-->}[rr]^{\Psi_{1}}\ar@{->}[dd] && M_X^{k_{1},l_{1}-m}(n,d-m)\ar@{<-}[ur]\ar@{->}[dd]\\
& \mathscr{B}un_X^{k_{2},l_{2}}(n, d)\ar@{-->}'[r][rr] & & \mathscr{B}un_X^{k_{2},l_{2}-m}(n,d-m)\\
M_X^{k_{2},l_{2}}(n,d)\ar@{-->}[rr]_{\Psi_{2}}\ar@{<-}[ur]&& M_X^{k_{2},l_{2}-m}(n,d-m)\ar@{<-}[ur]
}
$$
}

This diagram shows that the Hecke correspondence as constructed above determines a birational linear map between the moduli stacks of $(k,l)$-stable vector bundles as indicated by the uppermost dashed arrow.\\

{\it Acknowledgements:} Both authors would like to warmly thank the Centro de Investigaci\'on en Matem\'aticas (CIMAT) in Guanajuato for the wonderful hospitality and support. We like to thank especially Professor Leticia Brambila-Paz for many valuable comments and suggestions. The first author wants to thank CONACYT and the Universidad de Guadalajara for partial support and he would like to thank Professor Alexander Nesterov for his help and advice. The second author likes to acknowledge additional support from the University of Leicester via a Santander Travel Grant. Finally, both authors are grateful to the referee for many suggestions to improve this article.



\begin{thebibliography}{AAA}

\bibitem[Be1]{Be2} K. Behrend, {\em The Lefschetz trace formula for algebraic stacks}, Invent. Math. {\bf 112} (1993), 127--149.

\bibitem[Be2]{Be3} K. Behrend, {\em Derived $l$-adic categories for algebraic stacks}, Mem. Amer. Math. Soc. {\bf 163} (2003), no. 774, viii+93 pp.

\bibitem[BL]{BL-stra} L. Brambila-Paz, H. Lange,
{\em A stratification of the moduli space of vector bundles on curves},  J. Reine Angew. Math {\bf 494} (1988), 173--187.

\bibitem[BM]{Bra-Mat} L. Brambila-Paz, O. Mata-Gutiérrez, {\em On the Hilbert scheme of the moduli space of vector bundles over an algebraic curve}, Manuscripta Math. {\bf 142} (2013), 525-544.  DOI: 10.1007/s00229-013-0618-x

\bibitem[Dh]{AD} A. Dhillon, {\em On the chomology of moduli of vector bundles and the Tamawawa number of $SL_n$},
Canad. J. Math. Vol. {\bf 58} (2006), 1000--1025.

\bibitem[DreNa]{DreNa} J.-M. Drezet, M.S. Narasimhan, {\em Groupe de Picard des vari\'et\'es de modules de fibr\'es
semi-stables sur les courbes alg\'ebriques.}, Invent. Math. {\bf 97} (1989), 53--94.

\bibitem[Go1]{Tom} T. L. G\'omez, {\em Algebraic Stacks}, Proc. Indian Acad. Sci. (Math. Sci.) Vol. {\bf 111}, No. {\bf 1}, (2001), 1--31.

\bibitem[Go2]{TG} T. L. G\'omez, {\em Quantization of Hitchin´s Integrable System and the Geometric Langlands Conjecture,}
 in: A. Schmitt (Ed.): Affine Flag Manifolds and Principal Bundles. Trends in Mathematics, Springer Basel (2010), 51--90.

\bibitem[Gr]{Gro} A. Grothendieck et. al. {\em Revetements étales et groupe fondamental (SGA 1)},
Documents Mathématiques (Paris), 3 Socit\'e Mathematique de France,  Paris (2003), xviii+327 pp.

\bibitem[H]{Hei} J. Heinloth, {\em Lectures on the Moduli Stack of vector bundles on a curve},
in: A. Schmitt (Ed.): Affine Flag Manifolds and Principal Bundles. Trends in Mathematics, Springer Basel (2010), 123--153.

\bibitem[HS]{HS} J. Heinloth, A. H. W. Schmitt, {\em The cohomology rings of moduli stacks of principal bundles over curves},
Doc. Math. {\bf 15} (2010), 423--488.

\bibitem[Hi]{Hirs-pro} A. Hirschowitz, {\em Problèmes de Brill-Noether en rang sup\'erieur.}, C.R. Acad. Sci. Paris
{\bf 307} (1988), 153--156.

\bibitem[Ho]{NH} N. Hoffmann, {\em Moduli Stacks of Vector Bundles on Curves and the King-Schoefield Rationality Proof}, in:
F. Bogomolov and Y. Tschinkel (Eds.):{\em  Cohomological and Geometric Approaches to Rationality Problems},
New Perspectives. Progress in Mathematics 282, Birkh\"auser (2010), 133--148.

\bibitem[Hl]{YH} Y. I. Holla, {\em Counting maximal subbundles via Gromov-Witten invariants}, Math. Ann. {\bf 328} (2004), 121--133.

\bibitem[HL]{HL} D. Huybrechts, M. Lehn, {\em The Geometry of Moduli Spaces of Sheaves}, Aspects of Mathematics, E31 Friedr. Vieweg Sohn, Braunschweig, 1997.

\bibitem[L]{L-some} H. Lange, {\em Some Geometrical aspects of vector bundles on curves}, in: L. Brambila-Paz and X. G\'omez Mont (Eds.): {\em Topics in algebraic geometry (Guanajuato, 1989)},  Aportaciones Matem\'aticas No.{\bf 5} (1992), 53--74.

\bibitem[LN]{LN-Max} H. Lange, M.~S.~Narasimhan, {\em Maximal subbundles of rank two vector bundles on curves}, Math. Ann.
{\bf 266} (1983), 55--72.

\bibitem[LaOl1]{LO1} Y. Laszlo, M. Olsson, {\em The six operations for sheaves on Artin stacks I: finite coefficients}, Publ. Math. IHES {\bf 107} (2008), 109--168.

\bibitem[LaOl2]{LO2} Y. Laszlo, M. Olsson, {\em The six operations for sheaves on Artin stacks II: adic coefficients}, Publ. Math. IHES {\bf 107} (2008), 169--210.

\bibitem[LMB]{LMB} G. Laumon, L. Moret-Bailly, {\em Champs alg{\'e}briques}, Erg. der Math. Grenz. 3. Folge, Band {\bf 39}, Springer-Verlag, Berlin (2000).

\bibitem[LOZ]{LOZ} A.M. Levin, M.A. Olshanetsky,  A. Zotov, {\em Hitchin Systems-Symplectic Hecke Correspondence and Two Dimensional Version}, Commun. Math. Phys. {\bf 263} (2003), 93--133.


\bibitem[Li]{Lieblich} M. Lieblich, {\em Moduli of twisted sheaves}, Duke Math. J. {\bf 38} (2007), 23--118.

\bibitem[M]{Mat} O. Mata-Guti\'errez, {\em $(k,l)$-stable vector bundles on curves and Hecke Grassmannians}, to appear.

\bibitem[Mu]{M} D. Mumford, {\em Projective Invariants of projective structures and applications}, Proc. Internat. Congr. Mathematicians (Stockholm, 1962) 526--530.

\bibitem[MFK]{MFK} D. Mumford, J. Fogarty and F. Kirwan, {\em Geometric Invariant Theory}, Ergeb. der Math. Grenz. {\bf 34}, 3rd ed., Springer-Verlag, Berlin (1992).

\bibitem[NR1]{NR-def} M.~S.~Narasimhan, S.~Ramanan, {\em Deformations of the moduli space of vector bundles over an algebraic curve}, Ann. of Math. {\bf 101} (1975), 391--417.

\bibitem[NR2]{NR-Geo} M.~S. Narasimhan, S. Ramanan, {\em Geometry of Hecke cycles-1}, in: C.P. Ramanujam - A tribute.  Studies in Mathematics, No. {\bf 8}, Springer Verlag, (1978), 291--345.

\bibitem[NS]{NS} M.~S.~Narasimhan, C.~S. Seshadri, {\em Stable and Unitary vector bundles on a compact Riemann Surfaces}, Ann. of Math. {\bf 82} (1965), 540--567.

\bibitem[N]{Fra} F. Neumann, {\em Algebraic stacks and moduli of vector bundles}, IMPA Publicacoes Matematicas (IMPA Research Monographs) 2nd. ed., Rio de Janeiro (2011) .

\bibitem[NSt]{NSt} F. Neumann, U. Stuhler, {\em Moduli stacks of vector bundles and Frobenius morphisms}, in: R. Tandon (Ed.): Algebra and Number Theory, Proceedings of the Silver Jubilee Conference (Hyderabad, India, 2003), Delhi (2005), 126--146.

\bibitem[LeP]{LP} J. Le Potier, {\em Lectures on Vector Bundles}, Cambridge University Press, Cambridge (1997), 260 pp.

\bibitem[Ra]{Ra} S. Ramanan, {\em The moduli space of vector bundles on an algebraic curve}, Math. Ann. {\bf 200} (1973), 69--84.

\bibitem[RT]{Teix-ona} B. Russo, M. Teixidor i Bigas, {\em On a Conjecture of Lange}, J. Algebraic Geometry {\bf 8} (1999), 483--496.

\bibitem[Se]{Se} C.~S.~Seshadri, {\em Fibres vectoriels sur les courbes alg\'ebriques}, Ast\'erisque {\bf 96} Soci\'et\'e math\'ematique de France (1982).

\bibitem[Se2]{Se2} C.~S.~Seshadri, { \em Space of Unitary Vector Bundles on a Compact Riemann Surface}, Ann. of Math., {\bf 185}(2), 1967, 303--336.

\bibitem[S]{S} S.S. Shatz, {\em The decomposition and specialization of algebraic families of vector bundles}, Compositio Mathematica. {\bf35} (1997), 163--187.


\end{thebibliography}
\end{document}